\documentclass[reqno]{amsart}
\usepackage{amsfonts}
\usepackage{a4wide}
\usepackage{cite}
\usepackage[linktocpage]{hyperref}
\usepackage{cleveref}

\usepackage{amsmath,amssymb,amsthm,amsfonts}
\usepackage{mathrsfs}
\usepackage{bbm}

\usepackage{color}

\newtheorem{lemma}{Lemma}[section]
\newtheorem{theorem}{Theorem}[section]

\newtheorem{proposition}{Proposition}[section]
\newtheorem{remark}{Remark}[section]

\numberwithin{equation}{section}

\newcommand{\dis}{\displaystyle}

\newcommand{\rmi}{{\mathrm i}}
\newcommand{\rmre}{{\rm Re}}

\newcommand{\R}{\mathbb{R}}

\newcommand{\Z}{\mathbb{Z}}

\renewcommand{\S}{\mathbb{S}}
\newcommand{\T}{\mathbb{T}}

\newcommand{\FP}{\mathbf{P}}

\newcommand{\CE}{\mathcal{E}}

\newcommand{\ep}{\epsilon}

\newcommand{\na}{\nabla}

\newcommand{\al}{\alpha}
\newcommand{\be}{\beta}

\newcommand{\om}{\omega}
\newcommand{\la}{\lambda}
\newcommand{\de}{\delta}
\newcommand{\si}{\sigma}
\newcommand{\pa}{\partial}
\newcommand{\ka}{\kappa}
\newcommand{\eps}{\epsilon}

\newcommand{\De}{\Delta}
\newcommand{\Ga}{\Gamma}

\makeatletter
\@namedef{subjclassname@2020}{%
  \textup{2020} Mathematics Subject Classification}
\makeatother

\begin{document}
\title[Boltzmann equation modelling polyatomic gas]{Global bounded solutions to the Boltzmann equation for a polyatomic gas}

\author[R.-J. Duan]{Renjun Duan}
\address[RJD]{Department of Mathematics, The Chinese University of Hong Kong,
	Shatin, Hong Kong, P.R.~China}
\email{rjduan@math.cuhk.edu.hk}

\author[Z.-G. Li]{Zongguang Li}
\address[ZGL]{Department of Mathematics, The Chinese University of Hong Kong,
Shatin, Hong Kong, P.R.~China}
\email{zgli@math.cuhk.edu.hk}

\begin{abstract}
In this paper we consider the Boltzmann equation modelling the motion of a polyatomic gas where the integration collision operator in comparison with the classical one involves an additional internal energy variable $I\in\R_+$ and a parameter $\de\geq 2$ standing for the degree of freedom. In perturbation framework, we establish the global well-posedness  for bounded mild solutions near global equilibria on torus. The proof is based on the $L^2\cap L^\infty$ approach. Precisely, we first study the $L^2$ decay property for the linearized equation, then use the iteration technique for the linear integral operator to get the linear weighted $L^\infty$ decay, and in the end obtain $L^\infty$ bounds as well as exponential time decay of solutions for the nonlinear problem with the help of the Duhamel's principle. Throughout the proof, we present a careful analysis for treating the extra effect of internal energy variable $I$ and the parameter $\delta$.
\end{abstract}

%\date{\today}

\subjclass[2020]{35Q20, 35B35}

%35Q20  	Boltzmann equations
%35B35  	Stability in context of PDEs

\keywords{Boltzmann equation, polyatomic gas, global bounded solution, exponential stability}
\maketitle
\thispagestyle{empty}

\tableofcontents

\section{Introduction}
The classical Boltzmann equation is a fundamental mathematical model in collisional kinetic theory which describes the motion of a rarefied monatomic gas in non-equilibrium states. In comparison with the classical case, the corresponding models for the polyatomic gas have wider applications in engineering. In the situation where the internal energy is introduced by a continuous variable, initiated by \cite{BL}, Bourgat,  Desvillettes, Le Tallec and Perthame \cite{BDTP} and later Desvillettes \cite{Desvillettes} derived general models for polyatomic gases including the so-called Borgnakke-Larsen type model. Specifically, those works gave a precise description on how the degrees of freedom enters into the Boltzmann collision operation, showed that the collision at the molecule level is reversible and symmetric, and also proved the $H$-theorem. For a mixture of reactive gases, Desvillettes, Monaco and Salvarani \cite{DMS} proposed a kinetic model which can describe chemical reversible reactions. See also \cite{Berroir,Bird,MBKK,Pullin,Kuscer} and the references therein for more original models. 

In the paper, we consider the following Cauchy problem on the Boltzmann equation modelling the motion of a polyatomic gas
 \begin{eqnarray}\label{BE}
&\dis \pa_tF+v\cdot \na_x F=Q(F,F),   \quad &\dis F(0,x,v,I)=F_0(x,v,I),
\end{eqnarray}
where $F(t,x,v,I)\geq0$ is an unknown velocity  distribution function of particles with position $x\in \Omega=\T^3$, velocity $v\in \R^3$ and internal energy $I\in (0,\infty)$ at time $t> 0$ and initial data $F_0(x,v,I)$ is given. The collision operator $Q(F,G)$ that is local in $(t,x)$ takes the form of
\begin{align}\label{DefQ}
	Q(F,G)&=\int_{(\R^3)^3\times (\R_+)^3} W(v,v_*,I,I_*|v^\prime,v^\prime_*,I^\prime,I^\prime_*)\left(\frac{F^\prime G^\prime_*}{(I^\prime I^\prime_*)^{\de/2-1}}-\frac{FG_*}{(II_*)^{\de/2-1}}\right)dv_*dv^\prime dv^\prime_* dI_*dI^\prime dI^\prime_*\notag\\
	&=Q_+(F,F)-Q_-(F,F),
\end{align}
where $\de\in [2,\infty)$ is the internal degrees of freedom, $G_*=G(t,x,v_*,I_*)$, $F^\prime=F(t,x,v^\prime,I^\prime)$ and $G^\prime_*=G(t,x,v^\prime_*,I^\prime_*)$. Denoting $\de_3$ and $\de_1$ to be the Dirac's delta function in $\R^3$ and $\R$ respectively, we assume $W(v,v_*,I,I_*|v^\prime,v^\prime_*,I^\prime,I^\prime_*)$ satisfies 
\begin{multline}\label{DefW}
	W(v,v_*,I,I_*|v^\prime,v^\prime_*,I^\prime,I^\prime_*)=4(II_*)^{\de/2-1}\si(|v-v_*|,\big|\frac{v-v_*}{|v-v_*|}\cdot\frac{v^\prime-v^\prime_*}{|v^\prime-v^\prime_*|}\big|,I,I_*,I^\prime,I^\prime_*)\frac{|v-v_*|}{|v^\prime-v^\prime_*|}\\
	\times\de_3(v+v_*-v^\prime-v^\prime_*)\de_1(\frac{|v|^2}{2}+\frac{|v_*|^2}{2}+I+I_*-\frac{|v^\prime|^2}{2}-\frac{|v^\prime_*|^2}{2}-I^\prime-I^\prime_*),
\end{multline}
and
\begin{align}\label{ProW}
	&W(v,v_*,I,I_*|v^\prime,v^\prime_*,I^\prime,I^\prime_*)=W(v_*,v,I_*,I|v^\prime_*,v^\prime,I^\prime_*,I^\prime)\notag\\
	&W(v,v_*,I,I_*|v^\prime,v^\prime_*,I^\prime,I^\prime_*)=W(v^\prime,v^\prime_*,I^\prime,I^\prime_*|v,v_*,I,I_*)\notag\\
	&W(v,v_*,I,I_*|v^\prime,v^\prime_*,I^\prime,I^\prime_*)=W(v,v_*,I,I_*|v^\prime_*,v^\prime,I^\prime_*,I^\prime).
\end{align}
For convenience, we denote $u=v-v_*$, $u^\prime=v^\prime-v^\prime_*$, $\Phi=\frac{|u|^2}{4}+I+I_*$ and $\Phi^\prime=\frac{|u^\prime|^2}{4}+I^\prime+I^\prime_*$. Throughout the paper, we consider hard potentials with cut-off like models proposed in \cite{Bernhoff} which takes the form 
\begin{align}\label{si}
	\si=C\frac{\sqrt{|u|^2+4I+4I_*-4I^\prime-4I^\prime_*}}{|u|\Phi^{\de+(\al-1)/2}}(I^\prime I^\prime_*)^{\de/2-1},
\end{align} 
where $C$ is a constant and $\al\in[0,2]$. In fact, all results still hold if $\si$ is equivalent to the right hand side of \eqref{si}. Throughout the paper, we restrict to only the assumption \eqref{si} for convenience. If we make changes of variables $R=\frac{|v^\prime-v^\prime_*|}{4\Phi}$, $r=\frac{I^\prime}{(1-R)\Phi}$ and $\Phi^\prime=\frac{|u^\prime|^2}{4}+I^\prime+I^\prime_*$, we deduce from the definition of $Q$ \eqref{DefQ} that
\begin{align*}
	Q(F,G)&=\int_{\S^2\times [0,1]^2\times \R^3\times \R_+} B\left(\frac{F^\prime G^\prime_*}{(I^\prime I^\prime_*)^{\de/2-1}}-\frac{FG_*}{(II_*)^{\de/2-1}}\right)\notag\\
	&\qquad\times(r(1-r))^{\de/2-1}(1-R)^{\de-1}R^{1/2}(II_*)^{\de/2-1}d\omega dR dr dv^\prime_* dI^\prime_*,
\end{align*}
where
$$
B=C\,\Phi^{1-\al/2}=C\left(\frac{|v-v_*|^2}{4}+I+I_*\right)^{1-\al/2}.
$$
One may refer to \cite{Bernhoff} for detailed deduction.

According to \cite{Bernhoff}, the global equilibrium $M$ such that $Q(F,F)\equiv 0$ is given by
\begin{align*}
	M=M(v,I)=\frac{I^{\de/2-1}}{(2\pi)^{3/2}\Ga(\de/2)}e^{-|v|^2/2-I},
\end{align*}
where $\Ga(s)$ is the Gamma function defined by $\Ga(s)=\int_0^\infty x^{s-1}e^{-x}dx$. Due to the macroscopic conservation laws from collision invariants of $Q$ (cf.~\cite{Bernhoff}), we introduce as in \cite{Guo} the initial defect mass, momentum and energy, which are respectively defined by
\begin{align} 
	M_0:=&\int_{\T^3\times\R^3\times\R_+} \{F(t,x,v,I)-M(v,I) \}dxdvdI\notag\\=&\int_{\T^3\times\R^3\times\R_+} \{F_0(x,v,I)-M(v,I) \}dxdvdI, \label{M}\\
	J_0:=&\int_{\T^3\times\R^3\times\R_+} v\{F(t,x,v,I)-M(v,I) \}dxdvdI\notag\\=&\int_{\T^3\times\R^3\times\R_+} v\{F_0(x,v,I)-M(v,I) \}dxdvdI,\label{J}\\
    E_0:=&\int_{\T^3\times\R^3\times\R_+} (|v|^2+2I)\{F(t,x,v,I)-M(v,I) \}dxdvdI\notag\\=&\int_{\T^3\times\R^3\times\R_+} (|v|^2+2I)\{F_0(x,v,I)-M(v,I) \}dxdvdI\label{E}.
\end{align}

In what follows we recall some closely related mathematical literature for the polyatomic Boltzmann model. In the space homogeneous case, Gamba and Pavi\'c-{C}oli\'c \cite{GP} proved the global existence and uniqueness under the extended Grad assumption through a careful analysis on polynomial and exponential moments. In the space inhomogeneous case under the perturbation framework near global Maxwellian, Bernhoff \cite{Bernhoff} studied the linearized collision operator. The coercivity estimate of the collision frequency and compactness of integral operator are obtained by the pointwise estimates with the help of some lemmas as in \cite{Glassey,Drange}. See also \cite{BST,BST1} for the compactness and Fredholm properties of the linearized operator. Similar results were obtained by Bernhoff \cite{Bernhoff1} for the discrete internal energy variable model which provides another way to describe the internal energy. Interested readers may refer to \cite{CC,EG} for related models.

For the moment we also recall some related studies that have been done with the classical Boltzmann equation modelling a single monatomic gas. DiPerna and Lions \cite{DL} proved the global existence of renormalized solutions for general $L^1_{x,v}$ initial data with finite mass, energy and entropy; though, the uniqueness of such solutions remains unknown. To study the perturbation near global Maxwellians, there are extensive works on the linearized operator, see Grad \cite{Grad}, Ellis and Pinsky \cite{EP}, and Baranger and Mouhot \cite{BM}. Ukai \cite{Ukai} obtained the global classical solutions and large-time behavior through the semi-group theory. We also mention \cite{UA,NI,Shizuta,UY}. In Sobolev spaces, using the macro-micro decomposition, Liu, Yang and Yu \cite{LYY} and Guo \cite{Guo04} developed energy methods independently. For general $L^\infty$ solutions near vacuum, Kaniel and Shinbrot \cite{KS} first obtained unique local solutions and the global existence is proven by Illner and Shinbrot \cite{IS}. In general bounded domains, Guo \cite{GuoY} developed the well-posedness theory through an $L^2\cap L^\infty$ approach. When the spacial variable is in the whole space or torus, Guo \cite{Guo} proved that the smallness of $L^\infty$ norm can be replaced by the smallness of relative entropy in the existence and uniqueness theory. In the above results concerning the $L^\infty$ solutions, the amplitude of the initial data is required to be small, while \cite{DHWY} also established an $L^\infty_{x,v}\cap L^1_x L^\infty_v$ method to prove the well-posedness for a class of initial data allowing to have large oscillations. Later, the result was extended to $L^p_vL^\infty_TL^\infty_x$ spaces, where $p$ can be finite, by Nishimura \cite{Nishimura} and Li \cite{Li} for hard and soft potentials, respectively. Despite the massive studies on the classical Boltzmann equation, there are few results on the well-posedness theory for the Boltzmann model in the polyatomic case.

We set up the perturbation $f=f(t,x,v,I)$ by
\begin{align}\label{perturb}
	f=\frac{F-M}{\sqrt{M}}.
\end{align}
In the similar way as how we define $F_*$, we denote $M_*=M(v_*,I_*)$, $M^\prime=M(v^\prime,I^\prime)$,  $M^\prime_*=M(v^\prime_*,I^\prime_*)$, $f_*=f(t,x,v_*,I_*)$, $f^\prime=f(t,x,v^\prime,I^\prime)$ and $f^\prime_*=f(t,x,v^\prime_*,I^\prime_*)$.
Substituting \eqref{perturb} into \eqref{BE}, we rewrite the Cauchy problem \eqref{BE} as 
\begin{align}\label{PBE}
	\pa_tf+v\cdot \na_x f+L f=\Ga (f,f),   \quad &\dis f(0,x,v,I)=f_0(x,v,I).
\end{align}
The linear operator $L$ and the nonlinear term $\Ga$ are respectively given by
\begin{align*}%\label{DefL}
	Lf=-\frac{1}{\sqrt{M}}Q(M,\sqrt{M}f)-\frac{1}{\sqrt{M}}Q(\sqrt{M}f,M),
\end{align*}
and
\begin{align}\label{DefGa}
	\Ga(f,g)=\frac{1}{\sqrt{M}}Q(\sqrt{M}f,\sqrt{M}g).
\end{align}
We denote
\begin{align*}%\label{DefGa+-}
	\Ga_{\pm}(f,g)=\frac{1}{\sqrt{M}}Q_{\pm}(\sqrt{M}f,\sqrt{M}g).
\end{align*}
As in the monatomic case, one can write
%\begin{align*}%\label{DefL}
	$L=\nu-K$.
%\end{align*}
Here, $\nu$ is the collision frequency  written as
\begin{align*}
	\nu(v,I)=\int_{(\R^3)^3\times (\R_+)^3} W(v,v_*,I,I_*|v^\prime,v^\prime_*,I^\prime,I^\prime_*)\frac{M_*}{(I I_*)^{\de/2-1}}dv_*dv^\prime dv^\prime_* dI_*dI^\prime dI^\prime_*,
\end{align*}
and the integral operator $K=K_2-K_1$ is defined as
\begin{align}\label{DefK1}
	K_1f=\int_{(\R^3)^3\times (\R_+)^3} W(v,v_*,I,I_*|v^\prime,v^\prime_*,I^\prime,I^\prime_*)\frac{\sqrt{MM_*}}{(I I_*)^{\de/2-1}}f_*dv_*dv^\prime dv^\prime_* dI_*dI^\prime dI^\prime_*
\end{align}
and
\begin{align}\label{DefK2}
	K_2f=&\frac{1}{\sqrt{M}}\int_{(\R^3)^3\times (\R_+)^3} W(v,v_*,I,I_*|v^\prime,v^\prime_*,I^\prime,I^\prime_*)\notag\\
	&\qquad\qquad\qquad\times\left(\frac{M^\prime\sqrt{M^\prime_*}}{(I^\prime I^\prime_*)^{\de/2-1}}f^\prime_*+\frac{M^\prime_*\sqrt{M^\prime}}{(I^\prime I^\prime_*)^{\de/2-1}}f^\prime\right)dv_*dv^\prime dv^\prime_* dI_*dI^\prime dI^\prime_*.
\end{align}
From \eqref{PBE}, by integrating along the backward trajectory, we obtain the mild form
\begin{align}\label{mild}
	\dis f(t,x,v,I)=&e^{-\nu(v,I)t}f_0(x-vt,v,I)+\int_0^t e^{-\nu(v,I)(t-s)}(Kf)(s,x-v(t-s),v,I)ds \notag\\
	&+\int_0^t e^{-\nu(v,I)(t-s)}\Ga(f,f)(s,x-v(t-s),v,I)ds.
\end{align}

Given a function $f=f(t,x,v,I)$, the $L^\infty$ norm in $(x,v,I)$ is defined by
\begin{align*}
	&\|f(t)\|_\infty:=\sup_{x\in \T^3,v\in\R^3,I\in\R_+} |f(t,x,v,I)|.
\end{align*}
The two main results of the paper are stated below.

\begin{theorem}[Local existence]\label{local}
	Let $w(v,I)=(1+|v|+\sqrt{I})^\be$ with $\be > 5$. Assume $F_0(x,v,I)=M+\sqrt{M}f_0(x,v,I)\geq 0$ with $\|w f_0\|_{\infty}<\infty$. There is a positive constant $C_1>0$ which depends only in $\be$, $\al$ and $\de$, and a positive time $T_1>0$ where
	\begin{align}\label{T_1}
		T_1=\frac{1}{8C_1(1+\|w f_0\|_{\infty})}>0,
	\end{align}
	such that  the Cauchy problem \eqref{BE} has a unique mild solution $F(t,x,v,I)=M+\sqrt{M}f(t,x,v,I)\geq0$ for $(t,x,v,I)\in[0,T_1]\times \T^3 \times \R^3\times \R_+$  in the sense of \eqref{mild} satisfying
	\begin{align}\label{LE}
		\sup_{0\leq t \leq T_1}\left\|w 
		f(t)\right\|_{\infty}\leq 2\left\|w f_0\right\|_{\infty}.
	\end{align}
\end{theorem}

\begin{theorem}[Global existence]\label{global}
	Assume the same conditions as Theorem \ref{local}. In addition, suppose that the initial data satisfies $F_0(x,v,I)=M+\sqrt{M}f_0(x,v,I)\geq 0$  and the conservation laws \eqref{M}, \eqref{J} and \eqref{E} with $(M_0,J_0,E_0)=0$. Then there are constants $\eps>0$ and $\lambda>0$ such that if
	$
	\left\|w f_0\right\|_{\infty}\leq \eps, 
	$
	the Boltzmann equation \eqref{BE} has a unique global mild solution in the sense of \eqref{mild}, denoted by $F(t,x,v,I)=M+\sqrt{M}f(t,x,v,I)\geq0$ for $(t,x,v,I)\in \R_+\times \T^3 \times \R^3\times \R_+$, which satisfies
	\begin{align}\label{GE}
		\left\|w f(t)\right\|_{\infty}\leq Ce^{-\la t}\left\|w f_0\right\|_{\infty}
	\end{align}
	for any $t\geq 0$, where $C$ is a generic constant.
\end{theorem}

\begin{remark}
In the paper we follow the strategy of \cite{GuoY} and only focus on the global existence of mild solutions in $L^\infty$ framework instead of classical solutions in high-order Sobolev spaces as in \cite{Guo02,Guo03,Guo04}, because we start from the original collision operator of the form \eqref{DefQ} with $W$ involving two Dirac delta functions in \eqref{DefW}. However, if one uses another formulation with a smooth kernel like \cite{GP}, then it would be possible to use the energy method to work on establishing the existence of  more regular solutions assuming that initial data is  more regular.
\end{remark}

Now we give some notations which will be used later. Throughout the paper, if a constant $C$ depends on some parameters $\be_1,\be_2\cdots$, then we denote it by $C_{\be_1,\be_2,\cdots}$ to emphasize the explicit dependence. Without further explanation, we use $\chi$ as the characteristic function where $\chi_E(x)=1$ if $x\in E$ and $\chi_E(x)=0$ if $x\notin E$. Furthermore, we define some norms for given functions $f=f(t,x,v,I)$ and $g=g(t,x)$ as follows:
\begin{align*}
	&\|f(t)\|:=\left( \int_{\T^3\times\R^3\times\R_+} |f(t,x,v,I)|^2 dxdvdI\right)^{\frac{1}{2}}=\|f(t)\|_{L^2_{x,v,I}},\\
	&\|f(t,x)\|_{L^2_{v,I}}:=\left( \int_{\R^3\times\R_+} |f(t,x,v,I)|^2 dvdI\right)^{\frac{1}{2}},\\
	&\|g(t)\|:=\left( \int_{\T^3} |g(t,x)|^2 dx\right)^{\frac{1}{2}}.
\end{align*}
Also, for $g$, its Fourier transformation $\hat{g}$ is defined by
$$
\hat{g}(t,k)=\int_{\T^3}e^{-ix\cdot k}g(t,x)dx, \quad k\in \Z^3.
$$
%where $k\in \Z^3$.

We introduce the strategy of the proof. The $L^2\cap L^\infty$ framework we use in this paper is initially established in \cite{GuoY}. The major difficulty focuses on the additional continuous internal energy variable. In linear $L^\infty$ estimate, similar as in \cite{GuoY}, after double integration, the main problem is how to estimate the term
$$
\int_0^t e^{-\nu(v,I)(t-s)}\int k_w(v,v_*,I,I_*)\int_0^s e^{-\nu(v_*,I_*)(s-s_1)} K_wU(s_1)h_0ds_1dv_*dI_* ds.
$$
Notice that in the classical case, the corresponding term is similar as above except that  there is no $I$ and $I_*$.
Then when the velocity $|v|$ and relative velocities $|v-v_*|$ and $|v_*-v_{**}|$ are large, the above integral is small due to the property
$$
\int_{\R^3}k_w(v,v_*)e^{\frac{|v-v_*|^2}{20}}dv_* \leq C(1+|v|)^{-1}.
$$ 
Then we may reduce our problem to the case that all velocity variables are bounded, where the rest of the integral can be bounded by the $L^2$ norm. However, when there is one more unbounded internal energy variable, we will need other decay properties in the internal energy variables $I$ and $I_*$ to reduce the case to $L^2$ norms. Also, since our current weight function $w=w(v,I)$ contains $I$, we should give more precise estimates on the ratio $w(v,I)/w(v_*,I_*)$, which require us to use a series changes of variables and some properties of Dirac delta functions to absorb this term by some exponential decay in the explicit form of the integral operator. To overcome the difficulty mentioned above, we need the decay in large internal energies and it turns out that we do have it in Lemma \ref{lemmaK}. In fact, Lemma \ref{lemmaK} implies that even if ones multiply the integral kernel $k_w$ with some exponential growth term in the relative velocity $v-v_*$ and some polynomial growth term in the internal energy $I_*$, one can still obtain the decay in both velocity $v$ and internal energy $I$. Then we only need to focus on the case that all velocity and internal energy variables are bounded and using the Cauchy-Schwarz inequality we can turn the integral to $L^2$ norms of $f$.

For the linear $L^2$ decay, motivated by \cite{Duan,Guo02,LY,LYY}, we estimate the microscopic and macroscopic parts respectively. For the microscopic part, the coercivity estimate derived in \cite{Bernhoff} is important. We improve the inequality to get the dissipation in the microscopic part $P_2f$, where we have decomposed $f$ into $f=P_1f+P_2f$ with $P_1$ being the orthogonal projection to $\ker L$. For the dissipation of the macroscopic
part, we first write the equation for $P_1f$ in terms of the Fourier transformation, which is essentially a hyperbolic-parabolic system. Then we can apply Kawashima's technique \cite{Kawashima} to obtain the dissipation of $P_1f$, or equivalently $a$, $b$ and $c$. Combining the estimates of the microscopic and macroscopic parts, we construct a temporal Lyapunov functional which is equivalent with the desired total energy functional to derive the $L^2$ estimate.

For nonlinear problem, we use the solution operator of the linear problem to write the solution to the nonlinear equation using Duhamel's principle and double iteration. We point out that although we can prove for the nonlinear term $\Ga$ that  
$$
\left|w\Ga_+(f,f)\right|\leq C\nu(v,I)\|wf\|^2_\infty,
$$
our estimate on $K_w$ does not provide enough decay to cancel the $\nu$ above so we need to find other ways with good properties on $I$ to replace $\int k_w(v,v_*)(1+|v_*|)dv_* \leq C$ which is applied in \cite{GuoY,DHWY}. Therefore, the estimate on $\int^t_0\int^t_se^{-\nu(v,I)(t-s_1)}K_w\left\{S(s_1-s)w\Ga(\frac{h}{w},\frac{h}{w})(s)\right\}d{s_1}ds$ is different from the literature mentioned above. We seek an additional decay in $\nu$ from $e^{-\nu(v,I)(t-s_1)}$ in the integral, which needs a more careful analysis.

The rest of the paper is organized as follows. In Section 2 we prove some properties of the operators in the perturbed equation which will frequently appear in the next few sections. After finishing the preliminaries, there are four major steps to prove the global decay which is our Theorem \ref{global}. The local solution is constructed in Section 3 using a similar approximation sequence as in \cite{GuoY,DHWY}. We prove the boundedness of the approximation sequence and use it to deduce that it is a Cauchy sequence. Then, after taking the limit we obtain a local-in-time solution. Next, we should consider the linearized Boltzmann equation first. With some assumptions on the initial data including the conservation of mass, momentum and energy, we obtain the $L^2$ decay in Section 4. Then, in Section 5, we follow the strategy stated above to get the $L^\infty$ decay. By the Duhamel's principle, we deduce the exponential decay in time for the nonlinear problem in Section 6 to finish the proof.

\section{Properties of operators $L$ and $\Ga$}
In this section, we study the structure of the perturbed equation near global Maxwellians and prove useful properties of operators $\nu$, $K$ and $\Ga$ which will be frequently used in the following sections. The first Lemma we state below is about the equivalence between the collision frequency $\nu(v,I)$ and $(1+|v|+\sqrt{I})^{2-\al}$. We point out that the upper bound we obtain has removed the restriction of an additional parameter $\ep$  in \cite[Theorem 2]{Bernhoff} such that $\nu(v,I)$ and $\nu(v_*,I_*)$ are respectively equivalent to some polynomials with the same order and we can use Lemma 2.2 to bound $\int_{\R^3\times\R_+}k_w(v,v_*,I,I_*)\frac{\nu(v_*,I_*)}{\nu(v,I)}dv_*dI_*$ in the nonlinear estimate \eqref{J22} later.

\begin{lemma}\label{lenu}
	There exist constants $C$ and $\nu_0$ such that
	\begin{align}\label{estnu}
		\nu_0\leq\frac{1}{C}(1+|v|+\sqrt{I})^{2-\al}\leq \nu(v,I) \leq C(1+|v|+\sqrt{I})^{2-\al}.
	\end{align}
\end{lemma}
\begin{proof}
	Using the properties of Dirac delta function, denoting $V=\frac{v+v_*}{2}$, $V^\prime=\frac{v^\prime+v^\prime_*}{2}$, we first rewrite 
	\begin{align}\label{1nu}
		\nu(v,I)&=C\int_{(\R^3)^3\times (\R_+)^3} M_*\si \frac{|u|}{|u^\prime|}\de_3(v+v_*-v^\prime-v^\prime_*)\notag\\
		&\qquad\qquad\times\de_1(\frac{|v|^2}{2}+\frac{|v_*|^2}{2}+I+I_*-\frac{|v^\prime|^2}{2}-\frac{|v^\prime_*|^2}{2}-I^\prime-I^\prime_*)dv_*dv^\prime dv^\prime_* dI_*dI^\prime dI^\prime_*\notag\\
		&=C\int_{(\R^3)^3\times (\R_+)^3} M_*\si \chi_{\{|u|^2+4I+4I_*-4I^\prime-4I^\prime_*>0\}} \frac{|u|}{|u^\prime|^2}\de_3(V-V^\prime)\notag\\
		&\qquad\qquad\times\de_1(\sqrt{|u|^2+4I+4I_*-4I^\prime-4I^\prime_*}-|u^\prime|)dv_*du^\prime dV^\prime dI_*dI^\prime dI^\prime_*.
	\end{align}
	Letting $\om=\frac{u^\prime}{|u^\prime|}$ and using \eqref{si}, we have
		\begin{align}\label{2nu}
		\nu(v,I)&=C\int_{\R^3\times\R_+\times\S^2\times\R^3\times (\R_+)^3} M_*\frac{\sqrt{\Phi-(I^\prime+I^\prime_*)}}{\Phi^{\de+(\al-1)/2}}(I^\prime I^\prime_*)^{\de/2-1} \de_3(V-V^\prime)\notag\\
		&\ \times\de_1(\sqrt{|u|^2+4I+4I_*-4I^\prime-4I^\prime_*}-|u^\prime|)\chi_{\{|u|^2+4I+4I_*-4I^\prime-4I^\prime_*>0\}}dv_*d|u^\prime|d\om dV^\prime dI_*dI^\prime dI^\prime_*\notag\\
		&=C\int_{\R^3\times (\R_+)^3} M_*\frac{\sqrt{\Phi-(I^\prime+I^\prime_*)}}{\Phi^{\de+(\al-1)/2}}(I^\prime I^\prime_*)^{\de/2-1} \chi_{\{I^\prime+I^\prime_*<\Phi\}}dv_* dI_*dI^\prime dI^\prime_*\notag\\
		&=C\int_{\R^3\times (\R_+)^3} e^{-|v_*|^2/2-I_*}\frac{\sqrt{\Phi-(I^\prime+I^\prime_*)}}{\Phi^{\de+(\al-1)/2}}(I_*I^\prime I^\prime_*)^{\de/2-1} \chi_{\{I^\prime+I^\prime_*<\Phi\}}dv_* dI_*dI^\prime dI^\prime_*.
	\end{align}
	We now prove the lower bound of $\nu(v,I)$. Observing that $\chi_{\{I^\prime+I^\prime_*<\Phi\}}\geq \chi_{\{I^\prime<\Phi/4\}}\chi_{\{I^\prime_*<\Phi/4\}}$ and $\left(\Phi-(I^\prime+I^\prime_*)\right)\chi_{\{I^\prime<\Phi/4\}}\chi_{\{I^\prime_*<\Phi/4\}}\geq \Phi/2$, one has
	\begin{align*}
		\nu(v,I)&\geq C\int_{\R^3\times (\R_+)^3} e^{-|v_*|^2/2-I_*}\frac{\Phi^{1/2}}{\Phi^{\de+(\al-1)/2}}(I_*I^\prime I^\prime_*)^{\de/2-1} \chi_{\{I^\prime<\Phi/4\}}\chi_{\{I^\prime_*<\Phi/4\}}dv_* dI_*dI^\prime dI^\prime_*\notag\\
		&\geq C\int_{\R^3\times \R_+} e^{-|v_*|^2/2-I_*}\frac{\Phi^{1/2}}{\Phi^{\de+(\al-1)/2}}I_*^{\de/2-1}\left(\int_{\R_+}(I^\prime )^{\de/2-1}\chi_{\{I^\prime<\Phi/4\}}dI^\prime\right)^2 dv_* dI_* \notag\\
		&= C\int_{\R^3\times \R_+} e^{-|v_*|^2/2-I_*}I_*^{\de/2-1}\Phi^{\frac{2-\al}{2}} dv_* dI_*.
	\end{align*}
    The definition of $\Phi=\frac{|u|^2}{4}+I^\prime+I^\prime_*$ yields 
    \begin{align*}
    	\nu(v,I)&\geq C\int_{\R^3\times \R_+} e^{-|v_*|^2/2-I_*}I_*^{\de/2-1}(|v-v_*|+\sqrt{I}+\sqrt{I_*})^{2-\al} dv_* dI_*\notag\\
    	&\geq C\int_{\R^3} e^{-|v_*|^2/2}(|v-v_*|+\sqrt{I})^{2-\al} dv_*.
    \end{align*}
    By the inequality $1\geq \chi_{\{|v|\geq 1\}}\chi_{\{|v_*|\leq 1/2\}}+\chi_{\{|v|< 1\}}\chi_{\{|v_*|> 2\}}$, we have
        \begin{align}\label{lower}
    	\nu(v,I)&\geq C\int_{\R^3\times \R_+} e^{-|v_*|^2/2-I_*}I_*^{\de/2-1}(|v-v_*|+\sqrt{I}+\sqrt{I_*})^{2-\al} dv_* dI_*\notag\\
    	&\geq C\int_{\R^3} e^{-|v_*|^2/2}(\big||v|-|v_*|\big|+\sqrt{I})^{2-\al}\left(\chi_{\{|v|\geq 1\}}\chi_{\{|v_*|\leq 1/2\}}+\chi_{\{|v|< 1\}}\chi_{\{|v_*|> 2\}}\right) dv_*\notag\\
    	&\geq C\int_0^{1/2} e^{-|v_*|^2/2}dv_*(|v|+\sqrt{I})^{2-\al}\chi_{\{|v|\geq 1\}}+C\int_2^{\infty}e^{-|v_*|^2/2}dv_*(1+\sqrt{I})^{2-\al}\chi_{\{|v|< 1\}}\notag\\
    	&\geq C(1+|v|+\sqrt{I})^{2-\al} .
    \end{align}
    The upper bound can be obtained in the similar way. Using $\chi_{\{I^\prime+I^\prime_*<\Phi\}}\leq \chi_{\{I^\prime\leq \Phi\}}\chi_{\{I^\prime_*\leq \Phi\}}$ and $\sqrt{\Phi-(I^\prime+I^\prime_*)}\leq \sqrt{\Phi}$, it holds that
    	\begin{align}\label{3nu}
    	\nu(v,I)&\leq C\int_{\R^3\times (\R_+)^3} e^{-|v_*|^2/2-I_*}\frac{\Phi^{1/2}}{\Phi^{\de+(\al-1)/2}}(I_*I^\prime I^\prime_*)^{\de/2-1} \chi_{\{I^\prime\leq \Phi\}}\chi_{\{I^\prime_*\leq \Phi\}}dv_* dI_*dI^\prime dI^\prime_*\notag\\
    	&\leq C\int_{\R^3\times \R_+} e^{-|v_*|^2/2-I_*}\frac{\Phi^{1/2}}{\Phi^{\de+(\al-1)/2}}I_*^{\de/2-1}\left(\int_{\R_+}(I^\prime )^{\de/2-1}\chi_{\{I^\prime<\Phi\}}dI^\prime\right)^2 dv_* dI_*\notag\\
    	&=C\int_{\R^3\times \R_+} e^{-|v_*|^2/2-I_*}I_*^{\de/2-1}(|v-v_*|+\sqrt{I}+\sqrt{I_*})^{2-\al} dv_* dI_*.
    \end{align}
    Clearly $|v-v_*|+\sqrt{I}+\sqrt{I_*}\leq (1+|v_*|)(1+\sqrt{I_*})(1+|v|+\sqrt{I})$, then
    \begin{align}\label{upper}
    	\nu(v,I)&\leq C(1+|v|+\sqrt{I})^{2-\al}\int_{\R_+} e^{-I_*}I_*^{\de/2-1}(1+\sqrt{I_*})^{2-\al}dI_*\int_{\R^3}e^{-|v_*|^2/2}(1+|v_*|)^{2-\al} dv_*\notag\\
    	&\leq C(1+|v|+\sqrt{I})^{2-\al} .
    \end{align}
Collecting \eqref{lower} and \eqref{upper}, we obtain the desired estimate \eqref{estnu}.
\end{proof}

In $L^2\cap L^\infty$ approach for polyatomic gases, we can bound our integral by $L^2$ norm in the case that all the velocity and internal energy variables are bounded, so when these variables are large, the $L^\infty$ norm should be small. The following lemma implies such property. But, different from the monatomic case, we need some kind of growth property in the integral for the internal energy $I_*$ which is an integral variable to take care of the case when $I_*$ is large. Thus, we deduce our lemma with a polynomial increase in $I_*$ where the order can be strictly positive in order to get the decay. 

\begin{lemma}\label{lemmaK}
	Let $K=K_2-K_1$ be defined in \eqref{DefK1} and \eqref{DefK2}. There exist functions $k_i=k_i(v,v_*,I,I_*)$, $i=1,2$, such that 
	$$
	K_if(t,x,v,I)=\int_{\R^3\times \R_+}k_i(v,v_*,I,I_*)f(t,x,v_*,I_*)dv_*dI_*.
	$$
	Moreover, if we define $k(v,v_*,I,I_*)=k_2(v,v_*,I,I_*)-k_1(v,v_*,I,I_*)$ and 
	\begin{align}\label{defkw}
		k_w(v,v_*,I,I_*)=k(v,v_*,I,I_*)\frac{w(v,I)}{w(v_*,I_*)},
	\end{align}
    for $\be\in \R$, $0\leq\ep\leq\frac{1}{64}$ and $0\leq m\leq 1/8$, we have the estimate
    \begin{align}\label{estk}
    	\int_{\R^3\times \R_+}k_w(v,v_*,I,I_*)e^{\ep|v-v_*|^2}(1+I_*)^{m}dv_*dI_*\leq \frac{C_{\be,\ep,m}}{1+|v|+I^{1/4}}.
    \end{align}
    \end{lemma}
\begin{proof}
	 We have directly from the definition of $K_1$ \eqref{DefK1} that
	\begin{align*}
		k_1(v,v_*,I,I_*)=\int_{(\R^3)^2\times (\R_+)^2} W(v,v_*,I,I_*|v^\prime,v^\prime_*,I^\prime,I^\prime_*)\frac{\sqrt{MM_*}}{(I I_*)^{\de/2-1}}dv^\prime dv^\prime_* dI^\prime dI^\prime_*.
	\end{align*}
    Using similar arguments as in \eqref{1nu}, we have
    \begin{align}\label{1k1}
    	k_1&(v,v_*,I,I_*)\leq C\int_{\times\R_+\times\S^2\times\R^3\times (\R_+)^2} e^{-\frac{|v|^2}{4}-\frac{|v_*|^2}{4}-\frac{I}{2}-\frac{I_*}{2}}\frac{1}{\Phi^{\de+(\al-2)/2}}(I I_*)^{\de/4-1/2}(I^\prime I^\prime_*)^{\de/2-1} \notag\\
    	&\ \times\de_3(V-V^\prime)\de_1(\sqrt{|u|^2+4I+4I_*-4I^\prime-4I^\prime_*}-|u^\prime|)\chi_{\{|u|^2+4I+4I_*-4I^\prime-4I^\prime_*>0\}}d|u^\prime|d\om dV^\prime dI^\prime dI^\prime_*.
    \end{align}
We notice that $\de_3(V-V^\prime)\de_1(\sqrt{|u|^2+4I+4I_*-4I^\prime-4I^\prime_*}-|u^\prime|)$ implies $\Phi=\Phi^\prime$ and $$e^{-\frac{|v|^2}{8}-\frac{|v_*|^2}{8}-\frac{I}{4}-\frac{I_*}{4}}=e^{-\frac{|v^\prime|^2}{8}-\frac{|v^\prime_*|^2}{8}-\frac{I^\prime}{4}-\frac{I^\prime_*}{4}}$$ in the above integral. Then for the $\Phi^\de$ in the integral on the right hand side of \eqref{1k1}, we have 
\begin{align}\label{1E}
	\Phi^\de=\Phi^{1/2}(\Phi^\prime)^{2(\de/2-1/4)}=(\frac{|u|^2}{4}+I+I_*)^{1/2}(\frac{|u^\prime|^2}{4}+I^\prime+I^\prime_*)^{2(\de/2-1/4)}\geq(I_*)^{1/2}(I^\prime I^\prime_*)^{\de/2-1/4}
\end{align}
Combining \eqref{1k1} and \eqref{1E}, one has
\begin{align}\label{pointk1}
	k_1(v,v_*,I,I_*)&\leq C\Phi^{(2-\al)/2}I ^{\de/4-1/2}(I_*)^{\de/4-1} e^{-\frac{|v|^2}{8}-\frac{|v_*|^2}{8}-\frac{I}{4}-\frac{I_*}{4}}\int_{(\R^3)^2\times (\R_+)^2} e^{-\frac{|v^\prime|^2}{8}-\frac{|v^\prime_*|^2}{8}-\frac{I^\prime}{4}-\frac{I^\prime_*}{4}}\notag\\
	&\quad \times(I^\prime I^\prime_*)^{-3/4}\de_3(V-V^\prime)\de_1(\sqrt{|u|^2+4I+4I_*-4I^\prime-4I^\prime_*}-|u^\prime|)d|u^\prime|d\om dV^\prime dI^\prime dI^\prime_*\notag\\
	&\leq C(|v-v_*|+\sqrt{I}+\sqrt{I_*})^{(2-\al)/2}I ^{\de/4-1/2}(I_*)^{\de/4-1} e^{-\frac{|v|^2}{8}-\frac{|v_*|^2}{8}-\frac{I}{4}-\frac{I_*}{4}}\notag\\
	&\leq C(I_*)^{\de/4-1}e^{-\frac{|v|^2}{16}-\frac{|v_*|^2}{16}-\frac{I}{8}-\frac{I_*}{8}}.
\end{align}
Then we obtain
\begin{align}\label{estk1}
	&\int_{\R^3\times \R_+}k_1(v,v_*,I,I_*)\frac{w(v,I)}{w(v_*,I_*)}e^{\ep|v-v_*|^2}(1+I_*)^{m}dv_*dI_*\notag\\
	&\leq C\int_{\R^3\times \R_+}(I_*)^{\de/4-1}e^{-\frac{|v|^2}{16}-\frac{|v_*|^2}{16}-\frac{I}{8}-\frac{I_*}{8}}\frac{w(v,I)}{w(v_*,I_*)}e^{\ep|v-v_*|^2}(1+I_*)^{1/8}dv_*dI_*\notag\\
	&\leq Ce^{-\frac{|v|^2}{32}-\frac{I}{16}}
\end{align}
since $(I_*)^{\de/4-1}$ is integrable near $0$ by $\de/4-1\geq -1/2$.

Next we turn to $k_2$ which needs to be treated more carefully. First we notice in the integral of the definition of $K_2$ \eqref{DefK2}, the term $W(v,v_*,I,I_*|v^\prime,v^\prime_*,I^\prime,I^\prime_*)$ contains $$\de_3(v+v_*-v^\prime-v^\prime_*)\de_1(\frac{|v|^2}{2}+\frac{|v_*|^2}{2}+I+I_*-\frac{|v^\prime|^2}{2}-\frac{|v^\prime_*|^2}{2}-I^\prime-I^\prime_*),$$ which implies
\begin{align*}
	\sqrt{M^\prime}={I^\prime}^{\de/4-1/2}e^{-\frac{|v^\prime|^2}{4}-\frac{I^\prime}{2}}={I^\prime}^{\de/4-1/2}e^{-\frac{|v|^2}{4}-\frac{I}{2}-\frac{|v_*|^2}{4}-\frac{I_*}{2}+\frac{|v^\prime_*|^2}{4}+\frac{I^\prime_*}{2}}=\frac{(I^\prime I^\prime_*)^{\de/4-1/2}}{(I I_*)^{\de/4-1/2}}\frac{\sqrt{MM_*}}{\sqrt{M^\prime_*}},
\end{align*}
and
\begin{align*}
	\sqrt{M^\prime_*}=\frac{(I^\prime I^\prime_*)^{\de/4-1/2}}{(I I_*)^{\de/4-1/2}}\frac{\sqrt{MM_*}}{\sqrt{M^\prime}}.
\end{align*}
Then we can rewrite 
\begin{align}\label{1reK2}
	K_2f=&\frac{1}{\sqrt{M}}\int_{(\R^3)^3\times (\R_+)^3} W(v,v_*,I,I_*|v^\prime,v^\prime_*,I^\prime,I^\prime_*)\frac{\sqrt{MM_*M^\prime M^\prime_*}}{(II_*I^\prime I^\prime_*)^{\de/4-1/2}}\notag\\
	&\qquad\qquad\qquad\qquad\qquad\times\left(\frac{f^\prime_*}{\sqrt{M^\prime_*}}+\frac{f^\prime}{\sqrt{M^\prime}}\right)dv_*dv^\prime dv^\prime_* dI_*dI^\prime dI^\prime_*.
\end{align}
For the first part on the right hand side of \eqref{1reK2} containing $f^\prime_*$, we use the third equality of \eqref{ProW}, exchange $(v^\prime,I^\prime)$ and $(v^\prime_*,I^\prime_*)$, then exchange $(v^\prime,I^\prime)$ and $(v_*,I_*)$ to get 
\begin{align}\label{1partK2}
	&\frac{1}{\sqrt{M}}\int_{(\R^3)^3\times (\R_+)^3} W(v,v_*,I,I_*|v^\prime,v^\prime_*,I^\prime,I^\prime_*)\frac{\sqrt{MM_*M^\prime M^\prime_*}}{(II_*I^\prime I^\prime_*)^{\de/4-1/2}}\frac{f^\prime_*}{\sqrt{M^\prime_*}}dv_*dv^\prime dv^\prime_* dI_*dI^\prime dI^\prime_*\notag\\
	&=\frac{1}{\sqrt{M}}\int_{(\R^3)^3\times (\R_+)^3} W(v,v^\prime,I,I^\prime|v_*,v^\prime_*,I_*,I^\prime_*)\frac{\sqrt{MM_*M^\prime M^\prime_*}}{(II_*I^\prime I^\prime_*)^{\de/4-1/2}}\frac{f_*}{\sqrt{M_*}}dv_*dv^\prime dv^\prime_* dI_*dI^\prime dI^\prime_*\notag\\
	&=\int_{(\R^3)^3\times (\R_+)^3} W(v,v^\prime,I,I^\prime|v_*,v^\prime_*,I_*,I^\prime_*)\frac{\sqrt{M^\prime M^\prime_*}}{(II_*I^\prime I^\prime_*)^{\de/4-1/2}}f_*dv_*dv^\prime dv^\prime_* dI_*dI^\prime dI^\prime_*.
\end{align}
Likewise, we exchange $(v^\prime,I^\prime)$ and $(v_*,I_*)$ to get 
\begin{align}\label{2partK2}
	&\frac{1}{\sqrt{M}}\int_{(\R^3)^3\times (\R_+)^3} W(v,v_*,I,I_*|v^\prime,v^\prime_*,I^\prime,I^\prime_*)\frac{\sqrt{MM_*M^\prime M^\prime_*}}{(II_*I^\prime I^\prime_*)^{\de/4-1/2}}\frac{f^\prime}{\sqrt{M^\prime}}dv_*dv^\prime dv^\prime_* dI_*dI^\prime dI^\prime_*\notag\\
	&=\int_{(\R^3)^3\times (\R_+)^3} W(v,v^\prime,I,I^\prime|v_*,v^\prime_*,I_*,I^\prime_*)\frac{\sqrt{M^\prime M^\prime_*}}{(II_*I^\prime I^\prime_*)^{\de/4-1/2}}f_*dv_*dv^\prime dv^\prime_* dI_*dI^\prime dI^\prime_*.
\end{align}
Combining \eqref{1reK2}, \eqref{1partK2} and \eqref{2partK2}, we rewrite
\begin{align*}%\label{reK2}
	K_2f&=2\int_{(\R^3)^3\times (\R_+)^3} W(v,v^\prime,I,I^\prime|v_*,v^\prime_*,I_*,I^\prime_*)\frac{\sqrt{M^\prime M^\prime_*}}{(II_*I^\prime I^\prime_*)^{\de/4-1/2}}f_*dv_*dv^\prime dv^\prime_* dI_*dI^\prime dI^\prime_*,
\end{align*}
which yields
\begin{align}\label{rek2}
	k_2(v,v_*,I,I_*)&=2\int_{(\R^3)^2\times (\R_+)^2} W(v,v^\prime,I,I^\prime|v_*,v^\prime_*,I_*,I^\prime_*)\frac{\sqrt{M^\prime M^\prime_*}}{(II_*I^\prime I^\prime_*)^{\de/4-1/2}}dv^\prime dv^\prime_* dI^\prime dI^\prime_*.
\end{align}
Denoting $\xi=v-v^\prime$, $\xi_*=v_*-v^\prime_*$, $\Psi=\frac{|\xi|^2}{4}+I+I^\prime$,  $\zeta=(v_*-v^\prime)\cdot\frac{u}{|u|}$ and $\De J=I_*+I^\prime_*-I-I^\prime$, we have from \eqref{DefW}, \eqref{si} and properties of Dirac delta function that
\begin{align}\label{repW}
	W(v,v^\prime,I,I^\prime|v_*,v^\prime_*,I_*,I^\prime_*)=C\frac{(II_*I^\prime I^\prime_*)^{\de/2-1}}{|u|}\frac{\chi_{\{|\xi|+4I+4I^\prime>4I_*+4I^\prime_*\}}}{\Psi^{\de+(\al-1)/2}}\de_1(\zeta-\frac{\De J}{|u|})\de_3(u+u^\prime).
\end{align}
We resolve
\begin{align}\label{resolve}
	v^\prime-v_*=\eta-\zeta n,\quad n=\frac{u}{|u|}.
\end{align}
It holds that $dv^\prime dv^\prime_*=du^\prime d\zeta d\eta$.
We have from \eqref{rek2} and \eqref{repW} that
\begin{align}\label{k2}
	k_2(v,v_*,I,I_*)&=C\int_{\R^3\times\R_+\times(\R^3)^{\perp n}\times (\R_+)^2}\sqrt{M^\prime M^\prime_*} \frac{(II_*I^\prime I^\prime_*)^{\de/4-1/2}}{|u|}\frac{\chi_{\{|\xi|+4I+4I^\prime>4I_*+4I^\prime_*\}}}{\Psi^{\de+(\al-1)/2}}\notag\\
	&\qquad\qquad\qquad\times\de_1(\zeta-\frac{\De J}{|u|})\de_3(u+u^\prime)du^\prime d\zeta d\eta dI^\prime dI^\prime_*\notag\\
	&=C\int_{\R^3\times\R_+\times(\R^3)^{\perp n}\times (\R_+)^2}e^{-\frac{|v^\prime|^2}{4}-\frac{|v^\prime_*|^2}{4}-\frac{I^\prime}{2}-\frac{I^\prime_*}{2}} \frac{(II_*)^{\de/4-1/2}(I^\prime I^\prime_*)^{\de/2-1}}{|u|}\notag\\
	&\qquad\qquad\qquad\times\frac{\chi_{\{|\xi|+4I+4I^\prime>4I_*+4I^\prime_*\}}}{\Psi^{\de+(\al-1)/2}}\de_1(\zeta-\frac{\De J}{|u|})\de_3(u+u^\prime)du^\prime d\zeta d\eta dI^\prime dI^\prime_*.
\end{align}
Observe that we have $\frac{|v|^2}{2}+\frac{|v^\prime|^2}{2}+I+I^\prime=\frac{|v_*|^2}{2}+\frac{|v^\prime_*|^2}{2}+I_*+I^\prime_*$ in the integral of \eqref{k2} due to the term $\de_1(\zeta-\frac{\De J}{|u|})\de_3(u+u^\prime)$. It holds that
\begin{align*}
	1+|v|^2+I&=1+|v_*|^2+I_*+|v|^2-|v_*|^2+I-I_*\notag\\
	&\leq C(1+|v_*|^2+I_*)(1+\big|\frac{|v|^2}{2}-\frac{|v_*|^2}{2}+I-I_* \big|)\notag\\
	&= C(1+|v_*|^2+I_*)(1+\big|\frac{|v^\prime|^2}{2}-\frac{|v^\prime_*|^2}{2}+I^\prime-I^\prime_* \big|),
\end{align*}
which yields
\begin{align}\label{west}
	\frac{w(v,I)}{w(v_*,I_*)}&\leq C(1+|v^\prime|+|v^\prime_*|+\sqrt{I^\prime}+\sqrt{I^\prime_*} )^\be.
\end{align}
Also by the two Dirac delta functions we have $v-v_*=v^\prime_*-v^\prime$, which implies
\begin{align}\label{eest}
	e^{\ep|v-v_*|^2}=e^{\ep|v^\prime-v^\prime_*|^2}
\end{align}
in the integral of \eqref{k2} .
Combining \eqref{k2}, \eqref{west} and \eqref{eest}, one has
\begin{align}\label{1k2}
	&k_2(v,v_*,I,I_*)\frac{w(v,I)}{w(v_*,I_*)}e^{\ep|v-v_*|^2}\leq Ck_2(v,v_*,I,I_*)e^{\frac{|v^\prime|^2}{32}+\frac{|v^\prime_*|^2}{32}+\frac{I^\prime}{32}+\frac{I^\prime_*}{32}}\notag\\
	&\leq C\int_{\R^3\times\R_+\times(\R^3)^{\perp n}\times (\R_+)^2}e^{-\frac{|v^\prime|^2}{8}-\frac{|v^\prime_*|^2}{8}-\frac{I^\prime}{4}-\frac{I^\prime_*}{4}} \frac{(II_*)^{\de/4-1/2}(I^\prime I^\prime_*)^{\de/2-1}}{|u|}\notag\\
	&\qquad\qquad\qquad\times\frac{\chi_{\{|\xi|+4I+4I^\prime>4I_*+4I^\prime_*\}}}{\Psi^{\de+(\al-1)/2}}\de_1(\zeta-\frac{\De J}{|u|})\de_3(u+u^\prime)du^\prime d\zeta d\eta dI^\prime dI^\prime_*.
\end{align}
Resolving $V=V_\shortparallel+V_\perp$, $V_{\shortparallel}=V\cdot n$ and denoting $\cos\theta=n\cdot\frac{u}{|u|}$, it holds that
\begin{align}\label{resolve1}
	\frac{1}{8}(|v^\prime|^2+|v^\prime_*|^2)=\frac{1}{4}\left( |V_\perp+\eta|^2+\frac{(|u|-2|v|\cos\theta+2\zeta)^2}{4}+\frac{|u|^2}{4} \right).
\end{align}
Since $\Psi=\frac{|\xi|^2}{4}+I+I^\prime=\frac{|\xi_*|^2}{4}+I_*+I^\prime_*$ by $\de_1(\zeta-\frac{\De J}{|u|})\de_3(u+u^\prime)$, we have
\begin{align}\label{Psi1}
	\Psi=\frac{|\xi|^2}{4}+I+I^\prime\geq C|\xi|\sqrt{\frac{|\xi|^2}{4}+I+I^\prime-I_*-I^\prime_*}=C|\xi||\xi_*|\geq C|\eta|^2,
\end{align}
and
\begin{align*}
	\Psi^{\de-1/2}\geq \phi_i(I,I_*,I^\prime,I^\prime_*),\quad i=1,2,3,4,
\end{align*}
where the functions
\begin{align*}%\label{Defphii}
	\phi_1(I,I_*,I^\prime,I^\prime_*)=&(II_*)^{\de/4-1/2}(I^\prime I^\prime_*)^{\de/4+1/4},\notag\\
	\phi_2(I,I_*,I^\prime,I^\prime_*)=&I^{\de/4-1/2}I_*^{3\de/4},\notag\\
	\phi_3(I,I_*,I^\prime,I^\prime_*)=&I^{3\de/4}I_*^{\de/4-1/2},\notag\\
	\phi_4(I,I_*,I^\prime,I^\prime_*)=&I^{3\de/4-5/4}I_*^{\de/4+3/4}.
\end{align*}
We set
\begin{align}\label{Defphi}
	\phi(I,I_*,I^\prime,I^\prime_*)=&\frac{1}{\phi_1(I,I_*,I^\prime,I^\prime_*)}\chi_{\{I\leq 1\}}\chi_{\{I_*\leq 1\}}+\frac{1}{\phi_2(I,I_*,I^\prime,I^\prime_*)}\chi_{\{I\leq 1\}}\chi_{\{I_*\geq 1\}}\notag\\
	&\qquad+\frac{1}{\phi_3(I,I_*,I^\prime,I^\prime_*)}\chi_{\{I\geq 1\}}\chi_{\{I_*\leq 1\}}+\frac{1}{\phi_4(I,I_*,I^\prime,I^\prime_*)}\chi_{\{I\geq 1\}}\chi_{\{I_*\geq 1\}}
\end{align}
to get
\begin{align}\label{Psi2}
	\frac{1}{\Psi^{\de-1/2}}\leq \phi(I,I_*,I^\prime,I^\prime_*).
\end{align}
It follows from \eqref{1k2}, \eqref{resolve1}, \eqref{Psi1} and \eqref{Psi2} that
\begin{align}\label{pointk2}
	&k_2(v,v_*,I,I_*)\frac{w(v,I)}{w(v_*,I_*)}e^{\ep|v-v_*|^2}\notag\\
	&\leq \frac{C}{|u|}\int_{\R^3\times\R_+\times(\R^3)^{\perp n}\times (\R_+)^2}e^{-\frac{(|u|-2|v|\cos\theta+2\zeta)^2+|u|^2}{16}}e^{-\frac{I^\prime}{4}-\frac{I^\prime_*}{4}} (II_*)^{\de/4-1/2}(I^\prime I^\prime_*)^{\de/2-1}\notag\\
	&\qquad\qquad\qquad\qquad\times\phi(I,I_*,I^\prime,I^\prime_*)\frac{e^{-|V_\perp+\eta|^2/4}}{|\eta|^\al}\de_1(\zeta-\frac{\De J}{|u|})\de_3(u+u^\prime)du^\prime d\zeta d\eta dI^\prime dI^\prime_*\notag\\
	&\leq \frac{C}{|u|}\int_{(\R^3)^{\perp n}\times (\R_+)^2}e^{-\frac{(|u|-2|v|\cos\theta+2\De J/|u|)^2+|u|^2}{16}}e^{-\frac{I^\prime}{4}-\frac{I^\prime_*}{4}} (II_*)^{\de/4-1/2}(I^\prime I^\prime_*)^{\de/2-1}\notag\\
	&\qquad\qquad\qquad\qquad\qquad\qquad\qquad\qquad\qquad\quad\times\phi(I,I_*,I^\prime,I^\prime_*)\frac{e^{-|V_\perp+\eta|^2/4}}{|\eta|^\al}d\eta dI^\prime dI^\prime_*\notag\\
	&\leq \frac{C}{|u|}\int_{(\R_+)^2}e^{-\frac{(|u|-2|v|\cos\theta+2\De J/|u|)^2+|u|^2}{16}}e^{-\frac{I^\prime}{4}-\frac{I^\prime_*}{4}} (II_*)^{\de/4-1/2}(I^\prime I^\prime_*)^{\de/2-1}\phi(I,I_*,I^\prime,I^\prime_*)dI^\prime dI^\prime_*.
\end{align}
We integrate to get
\begin{align}\label{k2we}
&\int_{\R^3\times\R_+}k_2(v,v_*,I,I_*)\frac{w(v,I)}{w(v_*,I_*)}e^{\ep|v-v_*|^2}(1+I_*)^mdv_*dI_*\notag\\
&\leq C\int_{(\R_+)^3\times\R^3}\frac{1}{|u|}e^{-\frac{(|u|-2|v|\cos\theta+2\De J/|u|)^2+|u|^2}{16}} \notag\\
&\qquad\qquad\qquad\qquad\times e^{-\frac{I^\prime}{4}-\frac{I^\prime_*}{4}}(II_*)^{\de/4-1/2}(I^\prime I^\prime_*)^{\de/2-1}(1+I_*)^\frac{1}{8}\phi(I,I_*,I^\prime,I^\prime_*)dI^\prime dI^\prime_*dI_*dv_*.
\end{align}
Since the term $\De J$ contains $I_*$, $I^\prime$ and $I^\prime_*$, it is difficult to integrate on these three variables. However, we can study the integral $$\int_{(\R_+)^3}e^{-\frac{I^\prime}{4}-\frac{I^\prime_*}{4}} (II_*)^{\de/4-1/2}(I^\prime I^\prime_*)^{\de/2-1}(1+I_*)^\frac{1}{8}\phi(I,I_*,I^\prime,I^\prime_*)dI^\prime dI^\prime_*dI_* $$
first. Recalling the definition of $\phi$ \eqref{Defphi}, we obtain
\begin{align}\label{int4I}
	\int&_{(\R_+)^3}e^{-\frac{I^\prime}{4}-\frac{I^\prime_*}{4}} (II_*)^{\de/4-1/2}(I^\prime I^\prime_*)^{\de/2-1}(1+I_*)^\frac{1}{8}\phi(I,I_*,I^\prime,I^\prime_*)dI^\prime dI^\prime_*dI_*\notag\\
	\leq& C\int_{(\R_+)^3}e^{-\frac{I^\prime}{4}-\frac{I^\prime_*}{4}}\left((I^\prime I^\prime_*)^{\de/4-5/4}\chi_{\{I\leq 1\}}\chi_{\{I_*\leq 1\}}+I_*^{-\de/2-3/8}(I^\prime I^\prime_*)^{\de/2-1}\chi_{\{I\leq 1\}}\chi_{\{I_*\geq 1\}}\right.\notag\\
	&\left.+I^{-\de/2-1/2}(I^\prime I^\prime_*)^{\de/2-1}\chi_{\{I\geq 1\}}\chi_{\{I_*\leq 1\}}+I^{-\de/2+3/4}I_*^{-9/8}(I^\prime I^\prime_*)^{\de/2-1}\chi_{\{I\geq 1\}}\chi_{\{I_*\geq 1\}} \right)dI^\prime dI^\prime_*dI_*\notag\\
	\leq& C\int_{\R_+}\left(\chi_{\{I\leq 1\}}\chi_{\{I_*\leq 1\}}+I_*^{-\de/2-3/8}\chi_{\{I\leq 1\}}\chi_{\{I_*\geq 1\}}\right.\notag\\
	&\left.\qquad\qquad\qquad+I^{-\de/2-1/2}\chi_{\{I\geq 1\}}\chi_{\{I_*\leq 1\}}+I^{-\de/2+3/4}I_*^{-9/8}\chi_{\{I\geq 1\}}\chi_{\{I_*\geq 1\}} \right)dI_*\notag\\
	\leq& \frac{C}{1+I^{1/4}}.
\end{align} 
The last inequality above holds since $\de/2-3/4\geq 1/4$. Now we consider \eqref{k2we}. First we have directly from \eqref{k2we} and \eqref{int4I} that
\begin{align}\label{k21}
	&\int_{\R^3\times\R_+}k_2(v,v_*,I,I_*)\frac{w(v,I)}{w(v_*,I_*)}e^{\ep|v-v_*|^2}(1+I_*)^\frac{1}{8}dv_*dI_*\notag\\
	&\leq C\int_{(\R_+)^3\times\R^3}\frac{1}{|u|}e^{-\frac{|u|^2}{16}}e^{-\frac{I^\prime}{4}-\frac{I^\prime_*}{4}} (II_*)^{\de/4-1/2}(I^\prime I^\prime_*)^{\de/2-1}(1+I_*)^\frac{1}{8}\phi(I,I_*,I^\prime,I^\prime_*)dI^\prime dI^\prime_*dI_*dv_*\notag\\
	&\leq \frac{C}{(1+I^{1/4})}.
\end{align}
Then if $|u|\geq |v|$,
\begin{align}\label{k22}
	&\int_{\R^3\times\R_+}k_2(v,v_*,I,I_*)\frac{w(v,I)}{w(v_*,I_*)}e^{\ep|v-v_*|^2}(1+I_*)^\frac{1}{8}dv_*dI_*\notag\\
	&\leq \frac{1}{|v|}\int_{(\R_+)^3\times\R^3}e^{-\frac{|u|^2}{16}}e^{-\frac{I^\prime}{4}-\frac{I^\prime_*}{4}} (II_*)^{\de/4-1/2}(I^\prime I^\prime_*)^{\de/2-1}(1+I_*)^\frac{1}{8}\phi(I,I_*,I^\prime,I^\prime_*)dI^\prime dI^\prime_*dI_*dv_*\notag\\
	&\leq \frac{C}{|v|(1+I^{1/4})}.
\end{align}
If $|u|\leq |v|$, we set $r=|u|$ and use the relation $v\cdot(v-v_*)=r|v|\cos\theta$ to get
\begin{align}\label{k23}
	&\int_{\R^3\times\R_+}k_2(v,v_*,I,I_*)\frac{w(v,I)}{w(v_*,I_*)}e^{\ep|v-v_*|^2}(1+I_*)^\frac{1}{8}dv_*dI_*\notag\\
	&\leq C\int_{(\R_+)^3}\left(\int_0^{|v|}\int_0^\pi re^{-\frac{(r-2|v|\cos\theta+2\De J/r)^2+r^2}{16}}\sin\theta d\theta dr\right)\notag\\
	&\qquad\qquad\qquad\qquad\times e^{-\frac{I^\prime}{4}-\frac{I^\prime_*}{4}} (II_*)^{\de/4-1/2}(I^\prime I^\prime_*)^{\de/2-1}(1+I_*)^\frac{1}{8}\phi(I,I_*,I^\prime,I^\prime_*)dI^\prime dI^\prime_*dI_*\notag\\
	&\leq \frac{C}{|v|}\int_{(\R_+)^3}\left(\int_0^{|v|}\int_{r+2\De J/r-2|v|}^{r+2\De J/r+2|v|} re^{-\frac{y^2+r^2}{16}}dy dr\right)\notag\\
	&\qquad\qquad\qquad\qquad\times e^{-\frac{I^\prime}{4}-\frac{I^\prime_*}{4}} \frac{(II_*)^{\de/4-1/2}(I^\prime I^\prime_*)^{\de/2-1}(1+I_*)^\frac{1}{8}}{\phi(I,I_*,I^\prime,I^\prime_*)}dI^\prime dI^\prime_*dI_*\notag\\
	&\leq \frac{C}{|v|(1+I^{1/4})}.
\end{align}
It follows from \eqref{k21}, \eqref{k22} and \eqref{k23} that
\begin{align}\label{estk2}
	&\int_{\R^3\times\R_+}k_2(v,v_*,I,I_*)\frac{w(v,I)}{w(v_*,I_*)}e^{\ep|v-v_*|^2}(1+I_*)^\frac{1}{8}dv_*dI_*\notag
\\	
	&\leq \frac{C}{1+|v|+I^{1/4}}.
\end{align}
We obtain \eqref{estk} by \eqref{estk1} and \eqref{estk2}. Then Lemma \ref{lemmaK} is proven.
\end{proof}

For the nonlinear operator $\Ga$, we have the following lemma.
\begin{lemma}\label{lebound}
	Recall the definition of $\Ga$ \eqref{DefGa}. We have
	\begin{align}\label{estGa}
		\left|w(v,I)\Ga(f,f)(t,x,v,I) \right|\leq C\nu(v,I)\|(wf)(t)\|^2_\infty.
	\end{align}
\end{lemma}
\begin{proof}
	We write
	\begin{align*}
		&w(v,I)\Ga(f,f)(t,x,v,I)\notag\\
		&=\frac{w(v,I)}{\sqrt{M}}\int_{(\R^3)^3\times (\R_+)^3} W(v,v_*,I,I_*|v^\prime,v^\prime_*,I^\prime,I^\prime_*)\notag\\&\qquad\qquad\qquad\qquad\qquad\times\left(\frac{\sqrt{M^\prime M^\prime_*}}{(I^\prime I^\prime_*)^{\de/2-1}}f^\prime f^\prime_*-\frac{\sqrt{MM_*}}{(II_*)^{\de/2-1}}ff_*\right)dv_*dv^\prime dv^\prime_* dI_*dI^\prime dI^\prime_*\notag\\
		&=w(v,I)\Ga_+(f,f)(t,x,v,I)-w(v,I)\Ga_-(f,f)(t,x,v,I).
	\end{align*}
We first prove \eqref{estGa} for $\Ga_-$. A direct calculation shows that
\begin{align}\label{estGa-1}
	\left|[w\Ga_-(f,f)](t,x,v,I)\right|\leq \|(wf)(t)\|^2_\infty \int_{(\R^3)^3\times (\R_+)^3} W(v&,v_*,I,I_*|v^\prime,v^\prime_*,I^\prime,I^\prime_*)\notag\\
	&\times\frac{\sqrt{M_*}}{(I I_*)^{\de/2-1}}dv_*dv^\prime dv^\prime_* dI_*dI^\prime dI^\prime_*.
\end{align}
Applying change of variables $(v^\prime,v^\prime_*)\rightarrow(u^\prime,V^\prime)$ and using similar arguments in \eqref{1nu}, \eqref{2nu} and \eqref{3nu}, it holds that
\begin{align}\label{estGa-2}
	&\int_{(\R^3)^3\times (\R_+)^3} W(v,v_*,I,I_*|v^\prime,v^\prime_*,I^\prime,I^\prime_*)\frac{\sqrt{M_*}}{(I I_*)^{\de/2-1}}dv_*dv^\prime dv^\prime_* dI_*dI^\prime dI^\prime_*\notag\\
	&\leq C\int_{\R^3\times (\R_+)^3} e^{-|v_*|^2/4-I_*/2}\frac{\sqrt{\Phi-(I^\prime+I^\prime_*)}}{\Phi^{\de+(\al-1)/2}}{I_*}^{\de/4-1/2}(I^\prime I^\prime_* )^{\de/2-1} \chi_{\{I^\prime+I^\prime_*<\Phi\}}dv_* dI_*dI^\prime dI^\prime_*\notag\\
	&\leq C\int_{\R^3\times \R_+} e^{-|v_*|^2/4-I_*/2}I_*^{\de/4-1/2}(|v-v_*|+\sqrt{I}+\sqrt{I_*})^{2-\al} dv_* dI_*\notag\\
	& \leq C \nu(v,I).
\end{align}
It follows from \eqref{estGa-1} and \eqref{estGa-2} that
	\begin{align}\label{estGa-}
	\left|w(v,I)\Ga_-(f,f)(t,x,v,I) \right|\leq C\nu(v,I)\|(wf)(t)\|^2_\infty.
\end{align}
For the gain term, notice, we have
\begin{align*}%\label{estweight}
	w(v,I)\leq C(1+|v|^2+I)^{\frac{\be}{2}}\leq C(1+|v^\prime|^2+I^\prime+|v^\prime_*|^2+I^\prime_*)^\frac{\be}{2}\leq Cw(v^\prime,I^\prime)w(v^\prime_*,I^\prime_*)
\end{align*}
 by the Dirac delta functions in the term $W(v,v_*,I,I_*|v^\prime,v^\prime_*,I^\prime,I^\prime_*)$.
Then we use the above observation and take the $L^\infty$ norm to get
\begin{align}\label{estGa+1}
	\left|[w\Ga_+(f,f)](t,x,v,I)\right|&\leq C\|(wf)(t)\|^2_\infty \int_{(\R^3)^3\times (\R_+)^3} W(v,v_*,I,I_*|v^\prime,v^\prime_*,I^\prime,I^\prime_*)\notag\\
	&\qquad\qquad\qquad\qquad\times\frac{\sqrt{M^\prime M^\prime_*}}{\sqrt{M}}\frac{1}{(I^\prime I^\prime_*)^{\de/2-1}}dv_*dv^\prime dv^\prime_* dI_*dI^\prime dI^\prime_*\notag\\
	&=C\|(wf)(t)\|^2_\infty \int_{(\R^3)^3\times (\R_+)^3} W(v,v_*,I,I_*|v^\prime,v^\prime_*,I^\prime,I^\prime_*)\notag\\
	&\qquad\qquad\quad \times e^{-|v_*|^2/4-I_*/2}\frac{1}{(II^\prime I^\prime_*)^{\de/4-1/2}}dv_*dv^\prime dv^\prime_* dI_*dI^\prime dI^\prime_*.
\end{align}
Again we use the change of variables to get  
\begin{align}\label{estGa+2}
	&\int_{(\R^3)^3\times (\R_+)^3} W(v,v_*,I,I_*|v^\prime,v^\prime_*,I^\prime,I^\prime_*)e^{-|v_*|^2/4-I_*/2}\frac{1}{(II^\prime I^\prime_*)^{\de/4-1/2}}dv_*dv^\prime dv^\prime_* dI_*dI^\prime dI^\prime_*\notag\\
	&\leq C\int_{\R^3\times (\R_+)^3} e^{-|v_*|^2/4-I_*/2}\frac{\sqrt{\Phi-(I^\prime+I^\prime_*)}}{\Phi^{\de+(\al-1)/2}}I^{\de/4-1/2}{I_*}^{\de/2-1}(I^\prime I^\prime_* )^{\de/4-1/2} \chi_{\{I^\prime+I^\prime_*<\Phi\}}dv_* dI_*dI^\prime dI^\prime_*\notag\\
	&\leq C\int_{\R^3\times \R_+} e^{-|v_*|^2/4-I_*/2}\frac{1}{\Phi^{\de+(\al-2)/2}}I^{\de/4-1/2}{I_*}^{\de/2-1}\notag\\
	&\qquad\qquad\qquad\qquad\qquad\qquad\qquad\qquad\times\left(\int_{\R_+}(I^\prime )^{\de/4-1/2}\chi_{\{I^\prime<\Phi/4\}}dI^\prime\right)^2  \chi_{\{I^\prime+I^\prime_*<\Phi\}}dv_* dI_*\notag\\
	&\leq C\int_{\R^3\times\R_+} e^{-|v_*|^2/4-I_*/2}I^{\de/4-1/2}{I_*}^{\de/2-1}\frac{1}{\Phi^{\de/2-1}}\Phi^{(2-\al)/2}dv_* dI_*.
\end{align}
Using $\Phi^{\de/2-1}\geq C(II_*)^{\de/4-1/2}$, one has
\begin{align}\label{estGa+3}
	&C\int_{\R^3\times\R_+} e^{-|v_*|^2/4-I_*/2}I^{\de/4-1/2}{I_*}^{\de/2-1}\frac{1}{\Phi^{\de/2-1}}\Phi^{(2-\al)/2}dv_* dI_*\notag\\
	&\leq C\int_{\R^3\times\R_+} e^{-|v_*|^2/4-I_*/2}{I_*}^{\de/4-1/2}\Phi^{(2-\al)/2}dv_* dI_*\notag\\
	&\leq C\nu(v,I).
\end{align}
Then by \eqref{estGa+1}, \eqref{estGa+2} and \eqref{estGa+3}, we obtain
\begin{align}\label{estGa+}
	\left|w(v,I)\Ga_+(f,f)(t,x,v,I) \right|\leq C\nu(v,I)\|(wf)(t)\|^2_\infty.
\end{align}
Thus, \eqref{estGa} follows from \eqref{estGa-} and \eqref{estGa+}.
\end{proof}

\section{Local-in-time Existence}
In this section, we prove Theorem \ref{local}. Construct the iteration sequence in the following way for $n=0,1,2,\dots$
\begin{align*}
	&\dis \{\pa_t+v\cdot \na_x\} F^{n+1}+F^{n+1}\int_{(\R^3)^3\times (\R_+)^3} W(v,v_*,I,I_*|v^\prime,v^\prime_*,I^\prime,I^\prime_*)\frac{F_*^n}{(II_*)^{\de/2-1}}dv_*dv^\prime dv^\prime_* dI_*dI^\prime dI^\prime_*\notag\\&=Q_+(F^{n},F^{n}), 
\end{align*}
$F(0,x,v,I)=F_0(x,v,I)$ and $F^0(t,x,v,I)=M(v,I)$.

Integrating along the characteristic, we get
\begin{align}\label{1stappro}
	\dis F^{n+1}(t,x,v,I)=&e^{-{\int^t_0g^n(\tau,x-v(t-\tau),v,I)d\tau}}F_0(x-vt,v,I) \notag\\
	&+\int_0^t e^{-{\int^t_sg^n(\tau,x-v(t-\tau),v,I)d\tau}}Q_+(F^n,F^n)(s,x-v(t-s),v,I)ds,
\end{align}
where 
\begin{align*}
	g^n(t,x,v,I)=\int_{(\R^3)^3\times (\R_+)^3} W(v,v_*,I,I_*|v^\prime,v^\prime_*,I^\prime,I^\prime_*)\frac{F_*^n}{(II_*)^{\de/2-1}}dv_*dv^\prime dv^\prime_* dI_*dI^\prime dI^\prime_*.
\end{align*}
The positivity is clear to be observed from \eqref{1stappro} since if $F^n\geq 0$, then $Q_+(F^n,F^n)\geq 0$. Now we rewrite the approximation sequence in terms of $f^n$. Substituting $F^n=M+\sqrt{M}f^n$ into \eqref{1stappro}, we have
\begin{align}\label{appro}
	\dis f^{n+1}(t,x,v,I)=&e^{-{\int^t_0g^n(\tau,x-v(t-\tau),v,I)d\tau}}f_0(x-vt,v,I) \notag\\
	&+\int_0^t e^{-{\int^t_sg^n(\tau,x-v(t-\tau),v,I)d\tau}}(Kf^n)(s,x-v(t-s),v,I)ds \notag\\
	&+\int_0^t e^{-{\int^t_sg^n(\tau,x-v(t-\tau),v,I)d\tau}}\Ga_+(f^n,f^n)(s,x-v(t-s),v,I)ds, 
\end{align}
and $$g^n(t,x,v,I)=\int_{(\R^3)^3\times (\R_+)^3} W(v,v_*,I,I_*|v^\prime,v^\prime_*,I^\prime,I^\prime_*)\frac{\left( M_*+\sqrt{M_*}f_*^n\right)}{(II_*)^{\de/2-1}}dv_*dv^\prime dv^\prime_* dI_*dI^\prime dI^\prime_*,$$
with $f^{n+1}(0,x,v)=f_0(x,v)$ and $f^0(t,x,v)=0$.

From \eqref{appro}, we have the following lemmas, which imply that our approximation sequence is bounded and also a Cauchy sequence.
\begin{lemma}
	For $\be>5$, we have
	\begin{align}\label{LE2}
		\dis \sup_{0\leq t\leq T_1}\left\|(w f^{n+1})(t)\right\|_{\infty}\leq 2\|w_\be f_0\|_{\infty},
	\end{align}
    where $T_1$ is defined in \eqref{T_1}.
\end{lemma}
\begin{proof}
	We assume $(t,x,v)\in[0,T]\times \T^3 \times \R^3\times \R_+$. A direct calculation shows that
	\begin{align}\label{estwfn1}
		|(w f^{n+1})(t,x,v,I)|&\leq \|w f_0\|_{\infty}+\int_0^t \left|w(v,I)(Kf^n)(s,x-v(t-s),v)\right|ds\notag\\
		&\qquad\qquad\qquad\qquad+\int_0^t \left|w_\beta(v)\Ga_+(f^n,f^n)(s,x-v(t-s),v,I)\right|ds\notag \\
		&=\|w f_0\|_{\infty}+G_1+G_2.
	\end{align}
It follows from Lemma \ref{lemmaK} that
\begin{align}\label{estG1}
	G_1&=\int_0^t \left|w(v,I)(Kf^n)(s,x-v(t-s),v)\right|ds\notag\\
	&=\int_0^t \left|\int_{\R^3\times \R_+}k_w(v,v_*,I,I_*)(wf^n)(s,x-v(t-s),v_*,I_*)dv_*dI_*\right|ds\notag\\
	&\leq \|wf^n\|_{\infty}\int_0^t \left|\int_{\R^3\times \R_+}k_w(v,v_*,I,I_*)dv_*dI_*\right|ds\notag\\
	&\leq CT\|wf^n\|_{\infty}.
\end{align}
For $G_2$, we first write
\begin{align}\label{rewGa+}
	w(v,I)\Ga_+(f^n,f^n)(t,x,v,I)=\frac{w(v,I)}{\sqrt{M}}&\int_{(\R^3)^3\times (\R_+)^3} W(v,v_*,I,I_*|v^\prime,v^\prime_*,I^\prime,I^\prime_*)\notag\\
	&\quad\times\frac{\sqrt{M^\prime M^\prime_*}}{(I^\prime I^\prime_*)^{\de/2-1}}(f^n)^\prime (f^n)^\prime_*dv_*dv^\prime dv^\prime_* dI_*dI^\prime dI^\prime_*.
\end{align}
Due to the Dirac delta functions in $W(v,v_*,I,I_*|v^\prime,v^\prime_*,I^\prime,I^\prime_*)$, we have 
\begin{align}\label{estw}
	w(v,I)&\leq w(v,I)\chi_{\{|v|^2/2+I\leq 2(|v^\prime|^2/2+I^\prime)\}}+w(v,I)\chi_{\{|v|^2/2+I\leq 2(|v^\prime_*|^2/2+I^\prime_*)\}}\notag\\
	&\leq C(w(v^\prime,I^\prime)+w(v^\prime_*,I^\prime_*)).
\end{align}
Combining \eqref{rewGa+} and \eqref{estw}, one has
\begin{align}\label{estwGa+1}
	&\left|w(v,I)\Ga_+(f^n,f^n)(t,x,v,I)\right|\notag\\
	&\leq\|wf^n\|^2_{\infty}\frac{C}{\sqrt{M}}\int_{(\R^3)^3\times (\R_+)^3} W(v,v_*,I,I_*|v^\prime,v^\prime_*,I^\prime,I^\prime_*)\frac{\sqrt{M^\prime M^\prime_*}}{(I^\prime I^\prime_*)^{\de/2-1}}\notag\\
	&\qquad\qquad\qquad\qquad\qquad\times\left(\frac{1}{w(v^\prime,I^\prime)}+\frac{1}{w(v^\prime_*,I^\prime_*)}\right)dv_*dv^\prime dv^\prime_* dI_*dI^\prime dI^\prime_*\notag\\
	&=\|wf^n\|^2_{\infty}(G_{21}+G_{22}).
\end{align}
We now estimate 
$$
G_{21}=\frac{C}{\sqrt{M}}\int_{(\R^3)^3\times (\R_+)^3} W(v,v_*,I,I_*|v^\prime,v^\prime_*,I^\prime,I^\prime_*)\frac{\sqrt{M^\prime M^\prime_*}}{(I^\prime I^\prime_*)^{\de/2-1}}\frac{1}{w(v^\prime,I^\prime)}dv_*dv^\prime dv^\prime_* dI_*dI^\prime dI^\prime_*.
$$
By exchanging $(v^\prime,I^\prime)$ and $(v^*,I^*)$, we rewrite
\begin{align*}%\label{reG21}
	G_{21}=C\int_{(\R^3)^3\times (\R_+)^3} W(v,v^\prime,I,I^\prime|v_*,v^\prime_*,I_*,I^\prime_*)\frac{\sqrt{M_* M^\prime_*}}{\sqrt{M}}\frac{1}{(I_* I^\prime_*)^{\de/2-1}}\frac{1}{w(v_*,I_*)}dv_*dv^\prime dv^\prime_* dI_*dI^\prime dI^\prime_*.
\end{align*}
Using \eqref{DefW}, \eqref{si} and the decomposition \eqref{resolve}, it holds that
\begin{align*}
	G_{21}&\leq C\int_{\R^3\times\R_+\times(\R^3)^{\perp n}\times (\R_+)^2\times\R^3\times\R_+}e^{-\frac{|v^\prime|^2}{4}-\frac{I^\prime}{2}} \frac{(II_*)^{\de/4-1/2}(I^\prime )^{\de/2-1}(I^\prime_* )^{\de/4-1/2}}{|u|w(v_*,I_*)}\notag\\
	&\qquad\qquad\qquad\times\frac{1}{\Psi^{\de+(\al-1)/2}}\de_1(\zeta-\frac{\De J}{|u|})\de_3(u+u^\prime)du^\prime d\zeta d\eta dI^\prime dI^\prime_*dv_*dI_*.
\end{align*}
By \eqref{Psi1} and
\begin{align*}
	\frac{1}{\Psi^{\de-1/2}}\leq \bar{\phi}(I,I_*,I^\prime,I^\prime_*),
\end{align*}
where
\begin{align}\label{Defbarphi}
	\bar{\phi}(I,I_*,I^\prime,I^\prime_*)=\frac{1}{(II_*)^{\de/4-1/2}(I^\prime_*)^{\de/2+1/2}}\chi_{\{I^\prime_*> 1\}}+\frac{1}{(II_*)^{\de/4-1/2}(I^\prime)^{\de/2-1/4}(I^\prime_*)^{3/4}}\chi_{\{I^\prime_*\leq 1\}},
\end{align}
we get
\begin{align}\label{1estG21}
	G_{21}&\leq C\int_{\R^3\times\R_+\times(\R^3)^{\perp n}\times (\R_+)^2\times\R^3\times\R_+}e^{-\frac{I^\prime}{2}} \frac{(II_*)^{\de/4-1/2}(I^\prime )^{\de/2-1}(I^\prime_* )^{\de/4-1/2}}{|u|w(v_*,I_*)}\bar{\phi}(I,I_*,I^\prime,I^\prime_*)\notag\\
	&\qquad\qquad\qquad\times\frac{e^{-|v_*+\eta-\zeta n|^2/4}}{|\eta|^\al}\de_1(\zeta-\frac{\De J}{|u|})\de_3(u+u^\prime)du^\prime d\zeta d\eta dI^\prime dI^\prime_*dv_*dI_*\notag\\
	&\leq C\int_{\R^3\times\R_+}\left(\int_{(\R_+)^2}e^{-\frac{I^\prime}{2}}(II_*)^{\de/4-1/2}(I^\prime )^{\de/2-1}(I^\prime_* )^{\de/4-1/2}\bar{\phi}(I,I_*,I^\prime,I^\prime_*)dI^\prime dI^\prime_*\right)\notag\\
	&\qquad\qquad\qquad\times \frac{1}{|u|w(v_*,I_*)}dv_*dI_*.
\end{align}
From the definition of $\bar{\phi}$ \eqref{Defbarphi}, we have
\begin{align}\label{estbarphi}
	&\int_{(\R_+)^2}e^{-\frac{I^\prime}{2}}(II_*)^{\de/4-1/2}(I^\prime )^{\de/2-1}(I^\prime_* )^{\de/4-1/2}\bar{\phi}(I,I_*,I^\prime,I^\prime_*)dI^\prime dI^\prime_*\notag\\
	&\leq C\int_{(\R_+)^2}e^{-\frac{I^\prime}{2}}(I^\prime )^{\de/2-1}\frac{\chi_{\{I^\prime_*> 1\}}}{(I^\prime_* )^{\de/4+1}}dI^\prime dI^\prime_*+C\int_{(\R_+)^2}e^{-\frac{I^\prime}{2}}(I^\prime )^{-3/4}(I^\prime_* )^{\de/4-5/4}\chi_{\{I^\prime_*\leq 1\}}dI^\prime dI^\prime_*\notag\\
	&\leq C,
\end{align}
by the fact that $\de/4-5/4>-1$.
Collecting \eqref{1estG21} and \eqref{estbarphi}, it holds that
\begin{align}\label{estG21}
	G_{21}&\leq C\int_{\R^3\times\R_+} \frac{1}{|v-v_*|(1+|v_*|+\sqrt{I_*})^\be}dv_*dI_*\notag\\
	&\leq C\int_{\R^3\times\R_+} \frac{1}{|v-v_*|(1+|v_*|)^{3+(\be-5)/2} (1+\sqrt{I_*})^{(\be-1)/2}}dv_*dI_*\notag\\
	&\leq C.
\end{align}
The last inequality above holds since $(\be-5)/2>0$ and $(\be-1)/2>1$ by $\be>5$.

For $G_{22}$, we use the third equality of \eqref{ProW}, then exchange $(v^\prime,I^\prime)$ and $(v^\prime_*,I^\prime_*)$ to get
\begin{align}\label{estG22}
	G_{22}&=\frac{C}{\sqrt{M}}\int_{(\R^3)^3\times (\R_+)^3} W(v,v_*,I,I_*|v^\prime,v^\prime_*,I^\prime,I^\prime_*)\frac{\sqrt{M^\prime M^\prime_*}}{(I^\prime I^\prime_*)^{\de/2-1}}\frac{1}{w(v^\prime_*,I^\prime_*)}dv_*dv^\prime dv^\prime_* dI_*dI^\prime dI^\prime_*\notag\\
	&=\frac{C}{\sqrt{M}}\int_{(\R^3)^3\times (\R_+)^3} W(v,v_*,I,I_*|v^\prime_*,v^\prime,I^\prime_*,I^\prime)\frac{\sqrt{M^\prime M^\prime_*}}{(I^\prime I^\prime_*)^{\de/2-1}}\frac{1}{w(v^\prime_*,I^\prime_*)}dv_*dv^\prime dv^\prime_* dI_*dI^\prime dI^\prime_*\notag\\
	&=\frac{C}{\sqrt{M}}\int_{(\R^3)^3\times (\R_+)^3} W(v,v_*,I,I_*|v^\prime,v^\prime_*,I^\prime,I^\prime_*)\frac{\sqrt{M^\prime M^\prime_*}}{(I^\prime I^\prime_*)^{\de/2-1}}\frac{1}{w(v^\prime,I^\prime)}dv_*dv^\prime dv^\prime_* dI_*dI^\prime dI^\prime_*\notag\\
	&=G_{21}\notag\\
	&\leq C.
\end{align}
Then it follows from \eqref{estwGa+1}, \eqref{estG21} and \eqref{estG22} that
\begin{align}\label{estG2}
	G_2&=\left|w(v,I)\Ga_+(f^n,f^n)(t,x,v,I)\right|\leq CT\|wf^n\|^2_{\infty}.
\end{align}
Collecting \eqref{estwfn1}, \eqref{estG1} and \eqref{estG2}, one gets
\begin{align}\label{estwfn}
	|(w f^{n+1})(t,x,v,I)|\leq \|w f_0\|_{\infty}+C_1T(\|wf^n\|_{\infty}+\|wf^n\|^2_{\infty}),
\end{align}
for some constant $C_1>0$. 

Let $T_1$ be defined by \eqref{T_1} and assume $T\leq T_1$ and $\|wf^n\|_{\infty}\leq 2\|w f_0\|_{\infty}$. We have
\begin{align}\label{boundedness}
	|(w f^{n+1})(t,x,v,I)|\leq \|w f_0\|_{\infty}+6C_1T(\|wf_0\|_{\infty}+\|w f_0\|^2_{\infty})\leq 2\|w f_0\|_{\infty}.
\end{align}
Then Lemma \ref{lebound} is proved by induction.
\end{proof}

\begin{lemma}\label{leseq}
	For $\be>5$, we have
	\begin{align}\label{LE3}
		\dis \sup_{0\leq t\leq T_1}\left\|(\sqrt{w} f^{n+2}-\sqrt{w} f^{n+1})(t)\right\|_{\infty}\leq \frac{1}{2}\sup_{0\leq t\leq T_1}\left\|(\sqrt{w} f^{n+1}-\sqrt{w} f^{n})(t)\right\|_{\infty},
	\end{align}
	where $T_1$ is defined in \eqref{T_1}.
\end{lemma}
\begin{proof}
	For $(t,x,v)\in[0,T]\times \T^3 \times \R^3\times \R_+$, we take the difference between $\sqrt{w}f^{n+2}$ and $\sqrt{w}f^{n+1}$ to get
	\begin{align}\label{1diff}
		\dis &|(\sqrt{w}f^{n+2}-\sqrt{w}f^{n+1})(t,x,v,I)|\notag\\
		&\leq  |\sqrt{w(v,I)}f_0(x-vt,v,I)|\left| e^{-{\int^t_0g^{n+1}(\tau,x-v(t-\tau),v,I)d\tau}}-e^{-{\int^t_0g^n(\tau,x-v(t-\tau),v,I)d\tau}}  \right|\notag\\
		&\quad +\int_0^t \left|\sqrt{w(v,I)}\left(Kf^{n+1}\right)(s,x-v(t-s),v,I)\right|\notag \\
		& \qquad\qquad \times\left|e^{-{\int^t_sg^{n+1}(\tau,x-v(t-\tau),v,I)d\tau}}-e^{-{\int^t_sg^n(\tau,x-v(t-\tau),v,I)d\tau}}  \right| ds\notag\\
		&\quad +\int_0^t \left|\sqrt{w(v,I)}\Ga_+(f^{n+1},f^{n+1})(s,x-v(t-s),v,I)\right|\notag \\
		& \qquad\qquad \times\left|e^{-{\int^t_sg^{n+1}(\tau,x-v(t-\tau),v,I)d\tau}}-e^{-{\int^t_sg^n(\tau,x-v(t-\tau),v,I)d\tau}}\right|ds\notag \\
		&\quad +\int_0^t e^{-{\int^t_0g^n(\tau,x-v(t-\tau),v,I)d\tau}} \sqrt{w(v,I)}\left|(Kf^{n+1}-Kf^{n})(s,x-v(t-s),v,I)\right| ds \notag\\
		&\quad +\int_0^t e^{-{\int^t_0g^n(\tau,x-v(t-\tau),v,I)d\tau}} \sqrt{w(v,I)}\left|[\Ga_+(f^{n+1},f^{n+1})-\Ga_+(f^{n},f^{n})](s,x-v(t-s),v,I)\right| ds.
	\end{align}
By the fact that $|e^{-a}-e^{-b}|\leq|a-b|$ for any $a,b\geq0$, we have
\begin{align}\label{diffg}
&\left|e^{-{\int^t_sg^{n+1}(\tau,x-v(t-\tau),v,I)d\tau}}-e^{-{\int^t_sg^n(\tau,x-v(t-\tau),v,I)d\tau}}\right|\notag\\
&\leq \int^t_s\left| (g^{n+1}-g^n)(\tau,x-v(t-\tau),v) \right|d\tau\notag\\
&\leq T\sup_{0\leq t\leq T}\left\|(\sqrt{w} f^{n+1}-\sqrt{w} f^{n})(t)\right\|_{\infty}\notag\\
&\qquad\qquad\times\int_{(\R^3)^3\times (\R_+)^3} W(v,v_*,I,I_*|v^\prime,v^\prime_*,I^\prime,I^\prime_*)\frac{\sqrt{M_*}}{(I I_*)^{\de/2-1}}dv_*dv^\prime dv^\prime_* dI_*dI^\prime dI^\prime_*\notag\\
&\leq CT\nu(v,I)\sup_{0\leq t\leq T}\left\|(\sqrt{w} f^{n+1}-\sqrt{w} f^{n})(t)\right\|_{\infty}.
\end{align}
The last inequality above holds by \eqref{estGa-2}. Since $\sqrt{w(v,I)}\nu(v,I)\leq Cw(v,I)$ for $\be>5$, then from \eqref{1diff} and \eqref{diffg},
 \begin{align*}%\label{diff}
 	\dis &|(\sqrt{w}f^{n+2}-\sqrt{w}f^{n+1})(t,x,v,I)|\notag\\
 	&\leq  CT\|w f_0\|_{\infty}\sup_{0\leq t\leq T}\left\|(\sqrt{w} f^{n+1}-\sqrt{w} f^{n})(t)\right\|_{\infty}\notag\\
 	&\quad +CT\sup_{0\leq t\leq T}\left\|(\sqrt{w} f^{n+1}-\sqrt{w} f^{n})(t)\right\|_{\infty}\int_0^t \left|w(v,I)\left(Kf^{n+1}\right)(s,x-v(t-s),v,I)\right|ds\notag\\
 	&\quad +CT\sup_{0\leq t\leq T}\left\|(\sqrt{w} f^{n+1}-\sqrt{w} f^{n})(t)\right\|_{\infty}\int_0^t \left|w(v,I)\Ga_+(f^{n+1},f^{n+1})(s,x-v(t-s),v,I)\right|ds\notag \\
 	&\quad +\int_0^t  \sqrt{w(v,I)}\left|(Kf^{n+1}-Kf^{n})(s,x-v(t-s),v,I)\right| ds \notag\\
 	&\quad +\int_0^t  \sqrt{w(v,I)}\left|[\Ga_+(f^{n+1},f^{n+1})-\Ga_+(f^{n},f^{n})](s,x-v(t-s),v,I)\right| ds.
 \end{align*}
Using \eqref{estG1}, \eqref{estG2}, \eqref{estwfn} and similar arguments in \eqref{estG1}, we have
 \begin{align}\label{diff}
	\dis &|(\sqrt{w}f^{n+2}-\sqrt{w}f^{n+1})(t,x,v,I)|\notag\\
	&\leq  CT(1+\|w f_0\|_{\infty})\sup_{0\leq t\leq T}\left\|(\sqrt{w} f^{n+1}-\sqrt{w} f^{n})(t)\right\|_{\infty}\notag\\
	&\quad +\int_0^t  \sqrt{w(v,I)}\left|[\Ga_+(f^{n+1},f^{n+1})-\Ga_+(f^{n},f^{n})](s,x-v(t-s),v,I)\right| ds\notag\\
	&=CT(1+\|w f_0\|_{\infty})\sup_{0\leq t\leq T}\left\|(\sqrt{w} f^{n+1}-\sqrt{w} f^{n})(t)\right\|_{\infty}+G_3.
\end{align}
We split $G_3$ as follows:
\begin{align}\label{G3}
	G_3&=\int_0^t  \sqrt{w(v,I)}\left|[\Ga_+(f^{n+1},f^{n+1})-\Ga_+(f^{n},f^{n})](s,x-v(t-s),v,I)\right| ds\notag\\
	&\leq \int_0^t  \sqrt{w(v,I)}\left|\Ga_+(f^{n+1}-f^n,f^{n+1})(s,x-v(t-s),v,I)\right| ds\notag\\
	&\qquad\qquad\qquad+\int_0^t \sqrt{w(v,I)}\left|\Ga_+(f^{n},f^{n+1}-f^{n})(s,x-v(t-s),v,I)\right| ds\notag\\
	&=G_{31}+G_{32}.
\end{align}
Consider $G_{31}$ first. By \eqref{estw}, we obtain
\begin{align}\label{G311}
	G_{31}&=\int_0^t  \sqrt{w(v,I)}\left|\Ga_+(f^{n+1}-f^n,f^{n+1})(s,x-v(t-s),v,I)\right| ds\notag\\
	&\leq \frac{C}{\sqrt{M}}\int_0^t\int_{(\R^3)^3\times (\R_+)^3} W(v,v_*,I,I_*|v^\prime,v^\prime_*,I^\prime,I^\prime_*)\frac{\sqrt{M^\prime M^\prime_*}}{(I^\prime I^\prime_*)^{\de/2-1}}\frac{1}{w(v^\prime_*,I^\prime_*)}\notag\\
	&\qquad\qquad\qquad\times(\sqrt{w}f^{n+1}-\sqrt{w}f^n)^\prime (wf^{n+1})^\prime_*dv_*dv^\prime dv^\prime_* dI_*dI^\prime dI^\prime_*ds\notag\\
	&\quad+\frac{C}{\sqrt{M}}\int_0^t\int_{(\R^3)^3\times (\R_+)^3} W(v,v_*,I,I_*|v^\prime,v^\prime_*,I^\prime,I^\prime_*)\frac{\sqrt{M^\prime M^\prime_*}}{(I^\prime I^\prime_*)^{\de/2-1}}\frac{1}{\sqrt{w(v^\prime_*,I^\prime_*)}\sqrt{w(v^\prime,I^\prime)}}\notag\\
	&\qquad\qquad\qquad\times(\sqrt{w}f^{n+1}-\sqrt{w}f^n)^\prime (wf^{n+1})^\prime_*dv_*dv^\prime dv^\prime_* dI_*dI^\prime dI^\prime_*ds.
\end{align}
Using \eqref{estG22} for the first term on the right hand side of \eqref{G311} and the inequality $$\sqrt{w(v^\prime_*,I^\prime_*)}\sqrt{w(v^\prime,I^\prime)}\geq C\sqrt{w(v,I)},$$ similar arguments in \eqref{estGa+2} and \eqref{estGa+3} for the second term, one has
\begin{align}\label{G31}
	G_{31}&\leq CT\sup_{0\leq t\leq T}\left\|wf^{n+1}(t)\right\|_{\infty}\sup_{0\leq t\leq T}\left\|(\sqrt{w} f^{n+1}-\sqrt{w} f^{n})(t)\right\|_{\infty}\left(1+\frac{\nu(v,I)}{\sqrt{w(v,I)}}\right)\notag\\
	&\leq CT\left\|wf_0\right\|_{\infty}\sup_{0\leq t\leq T}\left\|(\sqrt{w} f^{n+1}-\sqrt{w} f^{n})(t)\right\|_{\infty}.
\end{align}
The last equality above holds by \eqref{LE2}.
Similarly we have
\begin{align}\label{G32}
	G_{32}&\leq CT\left\|wf_0\right\|_{\infty}\sup_{0\leq t\leq T}\left\|(\sqrt{w} f^{n+1}-\sqrt{w} f^{n})(t)\right\|_{\infty}.
\end{align}
Combining \eqref{diff}, \eqref{G3}, \eqref{G31} and \eqref{G32} and letting $T=T_1$, we conclude that
\begin{align*}
	\dis &|(\sqrt{w}f^{n+2}-\sqrt{w}f^{n+1})(t,x,v,I)|\notag\\
	&\leq  C_2T_1(1+\|w f_0\|_{\infty})\sup_{0\leq t\leq T}\left\|(\sqrt{w} f^{n+1}-\sqrt{w} f^{n})(t)\right\|_{\infty}\notag\\
	&\leq \frac{C_2}{8C_1}\sup_{0\leq t\leq T}\left\|(\sqrt{w} f^{n+1}-\sqrt{w} f^{n})(t)\right\|_{\infty}.
\end{align*}
Notice in \eqref{estwfn}, we can choose the constant $C_1$ to be sufficiently large. We then choose $C_1$ such that $\frac{C_2}{8C_1}\leq \frac{1}{2}$ to get the desired estimate \eqref{LE3}.
\end{proof}

Since $\{\sqrt{w}f^n\}$ is a Cauchy sequence, we can take the limit to obtain a local-in-time mild solution of \eqref{BE}. As a consequence \eqref{LE} follows from \eqref{boundedness}. The uniqueness can be obtained in a similar way to how we prove Lemma \ref{leseq} and the details are omitted fore brevity. This then completes the proof of Theorem \ref{local}. %\qed

From the next section, we extend the local solution to a global solution by the continuity argument.

\section{Linear $L^2$ decay}
In this section, we consider solving the linear problem for $f=f(t,x,v,I)$ in $L^2$ framework
\begin{align}\label{LBE}
	\pa_tf+v\cdot \na_x f+L f=0,
\end{align}
with the initial data $f(0,x,v,I)=f_0(x,v,I)$ satisfying \eqref{M}, \eqref{J} and \eqref{E}. 
By the properties of $W(v,v_*,I,I_*|v^\prime,v^\prime_*,I^\prime,I^\prime_*)$, we have the following proposition. Detailed proof can be found in Proposition 3 in \cite{Bernhoff}.
\begin{proposition}\label{propoker}
	The  linear operator $L$ is self-adjoint and nonnegative on $L^2(\R^3_v\times\R^+_I)$, and an orthogonal basis of $\ker L$ is given by
	\begin{align*}
		\left\{\sqrt{M},v\sqrt{M},(|v|^2+2I-3-\de)\sqrt{M}\right\}.
	\end{align*}
\end{proposition}
 The orthogonal projection operator to $\ker L$ is given by
\begin{align*}
	\FP_1f=\left\{a+b\cdot v+c(|v|^2+2I-3-\de)\right\}\sqrt{M},
\end{align*}
where $a=a(t,x)$, $b=b(t,x)=(b_1(t,x),b_2(t,x),b_3(t,x))$ and $c=c(t,x)$.
We set $\FP_2f=f-\FP_1f$.

By the pointwise estimates \eqref{pointk1} and \eqref{pointk2}, one can prove the compactness of $K$ using Lemma 3.5.1 in \cite{Glassey} or the contradiction argument as the Lemma 3 in \cite{Guo03}, details can be found in Section 4 in \cite{Bernhoff}. Moreover, together with Proposition \ref{propoker} and Lemma \ref{lenu}, we have the following proposition.

\begin{proposition}
	There exists a constant $\la_0>0$ such that for given function $f=f(v,I)$, we have
	\begin{align}\label{proco}
		\int_{\R^3\times\R_+}f(v,I)Lf(v,I)dvdI\geq C\la_0 \int_{\R^3\times\R_+} \nu(v,I)|\FP_2f(v,I)|^2dvdI\geq \la_0\|\FP_2f\|_{L^2_{v,I}}^2 
	\end{align}	
\end{proposition}
Then it is direct to obtain
\begin{align}\label{co}
	\frac{1}{2}\pa_t\|\hat{f}(t)\|^2_{L^2_{k,v,I}}+\la_0\|\FP_2\hat{f}(t)\|_{L^2_{k,v,I}}^2\leq0, 
\end{align}
from \eqref{LBE} and \eqref{proco}.

In order to deduce the $L^2$ estimate, we need to take care of terms including $\|Kf\|$, which needs the following lemma.

\begin{lemma}\label{L2K}
	There exists a constant C such that
	\begin{align*}
		\|Kf(t,x)\|_{L^2_{v,I}}\leq C\|f(t,x)\|_{L^2_{v,I}}.
	\end{align*}
\end{lemma}
\begin{proof}
	By Lemma \ref{lemmaK} and H\"older's inequality, one has
	\begin{align*}
		\int_{\R^3\times\R_+}|Kf(t,x,v,I)|^2dvdI&=\int_{\R^3\times\R_+}\left|\int_{\R^3\times\R_+}k(v,v_*,I,I_*)f(t,x,v_*,I_*)dv_*dI_*\right|^2dvdI\notag\\
		&\leq C\int_{\R^3\times\R_+}\left(\int_{\R^3\times\R_+}k(v,v_*,I,I_*)dv_*dI_*\right)\notag\\
		&\qquad\qquad\qquad\times\left(\int_{\R^3\times\R_+}k(v,v_*,I,I_*)\left|f(t,x,v_*,I_*)\right|^2dv_*dI_*\right)dvdI\notag\\
		&\leq C\int_{\R^3\times\R_+\times \R^3\times\R_+}k(v,v_*,I,I_*)\left|f(t,x,v_*,I_*)\right|^2dv_*dI_*dvdI\notag\\
		&\leq C\int_{\R^3\times\R_+}\left|f(t,x,v_*,I_*)\right|^2dv_*dI_*\notag\\
		&=C\|f(t,x)\|_{L^2_{v,I}}^2.
	\end{align*}
\end{proof}

With the above preparation, we will prove following the $L^2$ decay lemma.
\begin{lemma}\label{leL2}
	Let $f(t,x,v,I)$ be any solution to the linearized Boltzmann equation \eqref{LBE} with $(M_0, J_0, E_0) = (0,0,0)$. Then there exists positive constants $\la$ and $C$ such that
	\begin{align}\label{L2decay}
		\|f(t)\|\leq Ce^{-\la t}\|f_0\|,
	\end{align}
for $t\geq 0$.
\end{lemma}
\begin{proof}
Taking inner product of the linear equation \eqref{LBE} with the $5$ moments $\sqrt{M},v_i\sqrt{M},(|v|^2+2I-3-\de)\sqrt{M},(v_iv_j-1)\sqrt{M}$ and $(|v|^2+2I-5-\de)v_i\sqrt{M}$ for $i,j=1,2,3$, $i\neq j$ over $\R^3_v\times\R^+_I$ and using
\begin{align*}
    &\int_{\R^3}e^{-|v|^2/2}dv=1,\qquad \int_{\R^3}|v_i|^2e^{-|v|^2/2}dv=1, \qquad \int_{\R^3}|v|^2e^{-|v|^2/2}dv=3,\\
    &\int_{\R^3}|v_i|^2|v_j|^2e^{-|v|^2/2}dv=1, i\neq j,\qquad \int_{\R^3}|v_i|^4e^{-|v|^2/2}dv=3,\qquad \int_{\R^3}|v|^2|v_j|^2e^{-|v|^2/2}dv=5,\\
    &\int_{\R^3}|v|^4e^{-|v|^2/2}dv=15, \qquad \int_{\R^3}|v|^4|v_j|^2e^{-|v|^2/2}dv=35,\qquad\int_{\R^3}|v|^6e^{-|v|^2/2}dv=105,\\
    &\int_{R_+}I^{\de/2-1}e^{-I}dI=\Ga(\frac{\de}{2}),\qquad \int_{R_+}I^{\de/2}e^{-I}dI=\frac{\de}{2}\Ga(\frac{\de}{2}),\qquad \int_{R_+}I^{\de/2+1}e^{-I}dI=\frac{\de}{2}(\frac{\de}{2}+1)\Ga(\frac{\de}{2}),
\end{align*}
 we obtain the fluid-type system 
\begin{align*}%\label{fluid}
	\begin{cases}
		&\pa_{t}a+\nabla_{x}\cdot b=0,
		\\
		&\partial_{t}b_i+\pa_{x_i}(a+2c)+\nabla_{x}\cdot(\FP_2f,(v_iv-1)\sqrt{M})=0,
		\\
		&\pa_{t}c+\frac{1}{3+\de}\nabla_{x}\cdot b+\frac{1}{6+2\de}\nabla_{x}\cdot (\FP_2f,(|v|^2+2I-5-\de)v_i\sqrt{M})=0,
		\\
		&\pa_t \left( (\FP_2f,(v_iv_j-1)\sqrt{M})+2c\de_{ij} \right)+\pa_{xi}b_j+\pa_{xj}b_i\\
		&\quad\qquad=-(v\cdot\nabla_{x}\FP_2f,(v_iv_j-1)\sqrt{M})-(L\FP_2f,(v_iv_j-1)\sqrt{M}),
		\\&\pa_t (\FP_2f,(|v|^2+2I-5-\de)v_i\sqrt{M})+(10+2\de)\pa_{x_i}c\\
		&\quad\qquad=-(v\cdot\nabla_{x}\FP_2f,(|v|^2+2I-5-\de)v_i\sqrt{M})-(L\FP_2f,(|v|^2+2I-5-\de)v_i\sqrt{M}).
	\end{cases}
\end{align*}
Let $\hat{a}$, $\hat{b}$, $\hat{b}_i$, $\hat{c}$ and $\hat{f}$ denote the Fourier transformation of $a,b,b_i,c$ and $f$ respectively, then we rewrite the above system in terms of $\hat{a}$, $\hat{b}$, $\hat{c}$ and $\hat{f}$:
\begin{align}\label{syshat}
	\begin{cases}
		&\pa_{t}\hat{a}+\mathrm{i}k\cdot \hat{b}=0,
		\\
		&\partial_{t}\hat{b}_i+\mathrm{i}k_i(\hat{a}+2\hat{c})+\mathrm{i}k\cdot(\FP_2\hat{f},(v_iv-1)\sqrt{M})=0,
		\\
		&\pa_{t}\hat{c}+\frac{1}{3+\de}\mathrm{i}k\cdot \hat{b}+\frac{1}{6+2\de}\mathrm{i}k\cdot (\FP_2\hat{f},(|v|^2+2I-5-\de)v_i\sqrt{M})=0,
		\\
		&\pa_t \left( (\FP_2\hat{f},(v_iv_j-1)\sqrt{M})+2\hat{c}\de_{ij} \right)+\mathrm{i}k_i\hat{b}_j+\mathrm{i}k_j\hat{b}_i\\
		&\quad\qquad=-(\mathrm{i}k\cdot v\FP_2\hat{f},(v_iv_j-1)\sqrt{M})-(L\FP_2\hat{f},(v_iv_j-1)\sqrt{M}),
		\\&\pa_t (\FP_2\hat{f},(|v|^2+2I-5-\de)v_i\sqrt{M})+(10+2\de)\mathrm{i}k_i\hat{c}\\
		&\quad\qquad=-(\mathrm{i}k\cdot v\FP_2\hat{f},(|v|^2+2I-5-\de)v_i\sqrt{M})-(L\FP_2\hat{f},(|v|^2+2I-5-\de)v_i\sqrt{M}),
	\end{cases}
\end{align}
where $k\in \Z^3$ since $x\in \T^3$.
For convenience, we denote $p_{ij}(v,I)=(v_iv_j-1)\sqrt{M}$, $p_i(v,I)=(|v|^2+2I-5-\de)v_i\sqrt{M}$. Notice that both $p_{ij}$ and $p_i$ exponentially decay in $v$ and $I$.

We first estimate $\hat{c}$. From the fifth equation of \eqref{syshat}, we integrate by parts to get
\begin{align}\label{c1}
(10+2\de)|k|^2|\hat{c}|^2&=\sum^3_{i=1}\langle\rmi k_i\hat{c},(10+2\de)\rmi k_i\hat{c}\rangle\notag\\
&=\sum^3_{i=1}\left\{-\langle\rmi k_i\hat{c},\pa_t (\FP_2\hat{f},p_i(v,I))\rangle-\langle\rmi k_i\hat{c},-(\mathrm{i}k\cdot v\FP_2\hat{f},p_i(v,I))-(L\FP_2\hat{f},p_i(v,I))\rangle\right\}\notag\\
&=-\pa_{t}\sum^3_{i=1}\langle\rmi k_i\hat{c}, (\FP_2\hat{f},p_i(v,I))\rangle+\sum^3_{i=1}\langle\rmi k_i\pa_t\hat{c}, (\FP_2\hat{f},p_i(v,I))\rangle\notag\\
&\qquad-\sum^3_{i=1}\left\{\langle\rmi k_i\hat{c},-(\mathrm{i}k\cdot v\FP_2\hat{f},p_i(v,I))-(L\FP_2\hat{f},p_i(v,I))\rangle\right\}.
\end{align}
Since $p_i$ exponentially decays in $v$ and $I$, by Cauchy-Schwarz inequality, it holds that
\begin{align}\label{c2}
	\left|\langle\rmi k_i\hat{c},-(\mathrm{i}k\cdot v\FP_2\hat{f},p_i(v,I))\rangle\right|\leq \ep|k|^2|\hat{c}|^2+\frac{C}{\ep}(1+|k|^2)\|\FP_2\hat{f}\|_{L^2_{v,I}}^2.
\end{align} 
Similarly, by Lemma \ref{L2K} and Cauchy-Schwarz inequality,
\begin{align}\label{c3}
	\left|\langle\rmi k_i\hat{c},-(L\FP_2\hat{f},p_i(v,I))\rangle \right|\leq \ep|k|^2|\hat{c}|^2+\frac{C}{\ep}(1+|k|^2)\|\FP_2\hat{f}\|_{L^2_{v,I}}^2.
\end{align} 
Then we use the third equation in \eqref{syshat} to get
\begin{align*}
\langle\rmi k_i\pa_t\hat{c}, (\FP_2\hat{f},p_i(v,I))\rangle=-\langle\rmi k_i\left\{\frac{1}{3+\de}\mathrm{i}k\cdot \hat{b}+\frac{1}{6+2\de}\mathrm{i}k\cdot (\FP_2\hat{f},p_i(v,I))\right\},(\FP_2\hat{f},p_i(v,I))\rangle.	
\end{align*}
A direct application of Cauchy-Schwarz inequality yields
\begin{align}\label{c4}
	|\langle\rmi k_i\pa_t\hat{c}, (\FP_2\hat{f},p_i(v,I))\rangle|\leq \ep_1|k\cdot \hat{b}|^2+\frac{C}{\ep_1}(1+|k|^2)\|\FP_2\hat{f}\|_{L^2_{v,I}}^2.
\end{align}
Collecting \eqref{c1}, \eqref{c2}, \eqref{c3}, \eqref{c4} and choosing $\ep$ to be small enough, it holds that
\begin{align}\label{estc}
	\pa_{t}\rmre\sum^3_{i=1}\langle\rmi k_i\hat{c}, (\FP_2\hat{f},p_i(v,I))\rangle+\la|k|^2|\hat{c}|^2 \leq \ep_1|k\cdot \hat{b}|^2+\frac{C}{\ep_1}(1+|k|^2)\|\FP_2\hat{f}\|_{L^2_{v,I}}^2.
\end{align}
For the estimate of $\hat{b}$, we use the fourth equation in \eqref{syshat} to obtain
\begin{align}\label{b1}
	\sum^3_{i,j=1}|\rmi k_i\hat{b}_j+\rmi k_j\hat{b}_i|^2&=\sum^3_{i,j=1}\langle\rmi k_i\hat{b}_j+\rmi k_j\hat{b}_i,-\pa_t \left( (\FP_2\hat{f},p_{ij}(v,I))+2\hat{c}\de_{ij} \right)\rangle\notag\\
	&\qquad\qquad-\sum^3_{i,j=1}\langle(\rmi k\cdot v\FP_2\hat{f},p_{ij}(v,I))-(L\FP_2\hat{f},p_{ij}(v,I))\rangle.
\end{align}
Similar arguments as in \eqref{c2} and \eqref{c3} show that
\begin{align}\label{b2}
	\left|-\sum^3_{i,j=1}\langle(\rmi k\cdot v\FP_2\hat{f},p_{ij}(v,I))-(L\FP_2\hat{f},p_{ij}(v,I))\rangle\right|\leq \ep|k|^2|\hat{b}|^2+\frac{C}{\ep}(1+|k|^2)\|\FP_2\hat{f}\|_{L^2_{v,I}}^2.
\end{align}
It follows from integrating by parts that
\begin{align}\label{b3}
	&\langle\rmi k_i\hat{b}_j+\rmi k_j\hat{b}_i,-\pa_t \left( (\FP_2\hat{f},p_{ij}(v,I))+2c\de_{ij} \right)\rangle\notag\\
	&=-\pa_t 	\langle\rmi k_i\hat{b}_j+\rmi k_j\hat{b}_i,\left( (\FP_2\hat{f},p_{ij}(v,I))+2c\de_{ij} \right)\rangle+	\langle\pa_t (\rmi k_i\hat{b}_j+\rmi k_j\hat{b}_i),\left( (\FP_2\hat{f},p_{ij}(v,I))+2c\de_{ij} \right)\rangle.
\end{align}
Using the second equation in \eqref{syshat} and Cauchy-Schwarz inequality, it holds that
\begin{align}\label{b4}
	&\left|\langle\pa_t (\rmi k_i\hat{b}_j+\rmi k_j\hat{b}_i),\left( (\FP_2\hat{f},p_{ij}(v,I))+2\hat{c}\de_{ij} \right)\rangle\right|\notag\\
	&\leq C|k|\left|\langle \mathrm{i}k_i(\hat{a}+2\hat{c})+\mathrm{i}k\cdot(\FP_2\hat{f},(v_iv-1)\sqrt{M}),\left( (\FP_2\hat{f},p_{ij}(v,I))+2\hat{c}\de_{ij} \right)\rangle\right|\notag\\
	&\qquad+C|k|\left|\langle \mathrm{i}k_j(\hat{a}+2\hat{c})+\mathrm{i}k\cdot(\FP_2\hat{f},(v_jv-1)\sqrt{M}),\left( (\FP_2\hat{f},p_{ij}(v,I))+2\hat{c}\de_{ij} \right)\rangle\right|\notag\\
	&\leq \ep_2|k|^2|\hat{a}|^2+\frac{C}{\ep_2}|k|^2|\hat{c}|^2+\frac{C}{\ep_2}|k|^2\|\FP_2\hat{f}\|_{L^2_{v,I}}^2.
\end{align}
We notice that we have the identity
\begin{align*}
	\sum^3_{i,j=1}|\rmi k_i\hat{b}_j+\rmi k_j\hat{b}_i|^2=2|k|^2|\hat{b}|^2+2|k\cdot \hat{b}|^2.
\end{align*}
Together with \eqref{b1}, \eqref{b2}, \eqref{b3}, \eqref{b4} and choosing $\ep$ small, it holds that
\begin{align}\label{estb}
	&\pa_t \rmre	\sum^3_{i,j=1}\langle\rmi k_i\hat{b}_j+\rmi k_j\hat{b}_i,\left( (\FP_2\hat{f},p_{ij}(v,I))+2\hat{c}\de_{ij} \right)\rangle+\la|k|^2|\hat{b}|^2 \notag\\
	&\leq\ep_2|k|^2|\hat{a}|^2+\frac{C}{\ep_2}|k|^2|\hat{c}|^2+\frac{C}{\ep_2}|k|^2\|\FP_2\hat{f}\|_{L^2_{v,I}}^2.
\end{align}
Then we estimate $\hat{a}$. The second equation in \eqref{syshat} implies that
\begin{align}\label{a1}
	|k|^2|\hat{a}|^2&=\sum^3_{i=1}\langle\rmi k_i\hat{a},\rmi k_i\hat{a}\rangle\notag\\
	&=-\sum^3_{i=1}\left\{\langle\rmi k_i\hat{a},\partial_{t}\hat{b}_i+2\mathrm{i}k_i\hat{c}+\mathrm{i}k\cdot(\FP_2\hat{f},(v_iv-1)\sqrt{M})\rangle\right\}\notag\\
	&=-\pa_{t}\sum^3_{i=1}\langle\rmi k_i\hat{a},\hat{b}_i\rangle+\sum^3_{i=1}\langle\rmi k_i\pa_{t}\hat{a},\hat{b}_i\rangle\notag\\
	&\qquad\qquad\qquad\qquad-\sum^3_{i=1}\left\{\langle\rmi k_i\hat{a},2\mathrm{i}k_i\hat{c}+\mathrm{i}k\cdot(\FP_2\hat{f},(v_iv-1)\sqrt{M})\rangle\right\}.
\end{align}
Cauchy-Schwarz yields 
\begin{align}\label{a2}
	\left|-\langle\rmi k_i\hat{a},2\mathrm{i}k_i\hat{c}+\mathrm{i}k\cdot(\FP_2\hat{f},(v_iv-1)\sqrt{M})\rangle\right|\leq \ep|k|^2|\hat{a}|^2+\frac{C}{\ep}|k|^2|\hat{c}|^2+\frac{C}{\ep}|k|^2\|\FP_2\hat{f}\|_{L^2_{v,I}}^2.
\end{align}
From the first equation in \eqref{syshat}, we obtain
\begin{align}\label{a3}
	\left|\sum^3_{i=1}\langle\rmi k_i\pa_{t}\hat{a},\hat{b}_i\rangle\right|=\left|\sum^3_{i=1}\langle\rmi k_i(\mathrm{i}k\cdot \hat{b}),\hat{b}_i\rangle\right|=|k\cdot\hat{b}|^2.
\end{align}
By \eqref{a1}, \eqref{a2} and \eqref{a3}, choosing $\ep$ small, we get
\begin{align}\label{esta}
	\pa_{t}\rmre \sum^3_{i=1}\langle\rmi k_i\hat{a},\hat{b}_i\rangle+\la|k|^2|\hat{a}|^2\leq |k\cdot\hat{b}|^2+C|k|^2|\hat{c}|^2+C|k|^2\|\FP_2\hat{f}\|_{L^2_{v,I}}^2.
\end{align}
Multiplying \eqref{esta} by $\ep$ and taking sum with \eqref{estb}, one has
\begin{align*}
	&\pa_t \rmre	\sum^3_{i,j=1}\langle\rmi k_i\hat{b}_j+\rmi k_j\hat{b}_i,\left( (\FP_2\hat{f},p_{ij}(v,I))+2\hat{c}\de_{ij} \right)\rangle+\la|k|^2|\hat{b}|^2+\ep\pa_{t}\rmre \sum^3_{i=1}\langle\rmi k_i\hat{a},\hat{b}_i\rangle+\la\ep|k|^2|\hat{a}|^2\notag\\
	&\leq \ep|k\cdot\hat{b}|^2+C\ep|k|^2|\hat{c}|^2+C\ep|k|^2\|\FP_2\hat{f}\|_{L^2_{v,I}}^2+\ep_2|k|^2|\hat{a}|^2+\frac{C}{\ep_2}|k|^2|\hat{c}|^2+\frac{C}{\ep_2}|k|^2\|\FP_2\hat{f}\|_{L^2_{v,I}}^2.
\end{align*}
We first choose $\ep$ small, then choose $\ep_1$ small to get
\begin{align}\label{estab}
	&\pa_t \rmre\left(\sum^3_{i,j=1}\langle\rmi k_i\hat{b}_j+\rmi k_j\hat{b}_i,\left( (\FP_2\hat{f},p_{ij}(v,I))+2\hat{c}\de_{ij} \right)\rangle+\sum^3_{i=1}\langle\rmi k_i\hat{a},\hat{b}_i\rangle\right)+\la|k|^2(|\hat{a}|^2+|\hat{b}|^2)\notag\\
	&\leq C|k|^2|\hat{c}|^2+C|k|^2\|\FP_2\hat{f}\|_{L^2_{v,I}}^2.
\end{align}
Multiplying \eqref{estab} by $\ep$ and taking sum with \eqref{estc} one has
\begin{align}%\label{estab}
	&\pa_t \ep\rmre\left(\sum^3_{i,j=1}\langle\rmi k_i\hat{b}_j+\rmi k_j\hat{b}_i,\left( (\FP_2\hat{f},p_{ij}(v,I))+2\hat{c}\de_{ij} \right)\rangle+\sum^3_{i=1}\langle\rmi k_i\hat{a},\hat{b}_i\rangle\right)+\la\ep|k|^2(|\hat{a}|^2+|\hat{b}|^2)\notag\\
	&\qquad+\pa_{t}\rmre\sum^3_{i=1}\langle\rmi k_i\hat{c}, (\FP_2\hat{f},p_i(v,I))\rangle+\la|k|^2|\hat{c}|^2 \notag\\
	&\leq C\ep|k|^2|\hat{c}|^2+C\ep|k|^2\|\FP_2\hat{f}\|_{L^2_{v,I}}^2+\ep_1|k\cdot \hat{b}|^2+\frac{C}{\ep_1}(1+|k|^2)\|\FP_2\hat{f}\|_{L^2_{v,I}}^2.\notag
\end{align}
Choosing $\ep$ small first, then choosing $\ep_2$ small and taking integral with respect to $k$, it holds that
\begin{align}\label{estabc}
	&\pa_t \rmre\,  \CE^{int}(\hat{f})+\la\frac{|k|^2}{1+|k|^2}(\|\hat{a}(t)\|^2+\|\hat{b}(t)\|^2+\|\hat{c}(t)\|^2)\notag\\
	&\leq C\|\FP_2\hat{f}\|_{L^2_{k,v,I}}^2,
\end{align}
where
\begin{align*}%\label{defint}
\CE^{int}(\hat{f}):=&\int_{\T^3}\left\{\sum^3_{i,j=1}\langle\rmi k_i\hat{b}_j+\rmi k_j\hat{b}_i,\left( (\FP_2\hat{f},p_{ij}(v,I))+2\hat{c}\de_{ij} \right)\rangle\right.\notag\\
&\qquad\qquad\qquad\left.+\sum^3_{i=1}\langle\rmi k_i\hat{a},\hat{b}_i\rangle+\sum^3_{i=1}\langle\rmi k_i\hat{c}, (\FP_2\hat{f},p_i(v,I))\rangle\right\}dx.
\end{align*}
It is direct to see that 
$$
\left|\rmre\,\CE^{int}(\hat{f})\right|\leq C\|\hat{f}\|_{L^2_{k,v,I}}.
$$
Then we can choose $\ep$ small such that
\begin{align*}%\label{DefCE}
	\CE(\hat{f}):=\|\hat{f}\|_{L^2_{k,v,I}}^2+\ep\rmre\,\CE^{int}(\hat{f})\sim\|\hat{f}\|_{L^2_{k,v,I}}^2.
\end{align*}
We multiply \eqref{estabc} by a small number and take sum with \eqref{co} to get 
\begin{align}\label{estCE}
	\pa_{t}\CE(\hat{f}(t))+\la\CE(\hat{f}(t))\leq 0,
\end{align}
for $k\neq0$. In order to get exponential decay, the case when $k=0$ needs to be estimated separately.
From the assumption $(M_0, J_0, E_0) = (0,0,0)$, we obtain
\begin{align*}
	\int_{\T^3}a(0,x)dx=\int_{\T^3}b_i(0,x)dx=\int_{\T^3}c(0,x)dx=0,\ i=1,2,3,
\end{align*}
which yields
\begin{align*}
	\int_{\T^3}a(t,x)dx=\int_{\T^3}b_i(t,x)dx=\int_{\T^3}c(t,x)dx=0,\ i=1,2,3,
\end{align*}
by
$$
\frac{d}{dt}\int_{\T^3}(a,b,c)(t,x)dx=0
$$
from the first three equations in -\eqref{syshat}. Hence we have
$$
(\hat{a},\hat{b},\hat{c})(t,0)=\int_{\T^3}(a,b,c)(t,x)dx=0,
$$
which implies \eqref{estCE} for $k=0$. Then \eqref{L2decay} follows from \eqref{estCE} and Gronwall's inequality. Thus Lemma \ref{leL2} is proven.
\end{proof}

\section{Linear $L^\infty$ decay}

In this section, following the strategy in \cite{GuoY}, we consider the linear problem \eqref{LBE} in the weighted $L^\infty$ framework via the iteration technique as well as the interplay with $L^2$ properties. Recall $w(v,I)=(1+|v|+\sqrt{I})^\be$. Denote the function $h=h(t,x,v,I)=(wf)(t,x,v,I)$ and the operator $K_wh=wK\frac{h}{w}$. We rewrite the linearized Boltzmann equation in terms of $h$:
\begin{align}\label{LBEh}
	\pa_th+v\cdot \na_x h+\nu h-K_wh=0,\qquad h(0,x,v,I)=h_0(x,v,I)=(wf_0)(x,v,I).
\end{align}
Let $S(t)$ be the operator which solves \eqref{LBEh}, then the mild solution of \eqref{LBEh} is defined by
\begin{align}\label{mildlinear}
	\left\{S(t)h_0\right\}(t,x,v)&=h(t,x,v,I)\notag\\
	&=e^{-\nu(v,I)t}h_0(x-vt,v,I)+\int_0^t e^{-\nu(v,I)(t-s)}(K_wh)(s,x-v(t-s),v,I)ds.
\end{align}
Motivated by the $L^2\cap L^\infty$ approach developed by Guo \cite{GuoY}, we have the following $L^\infty$ decay property for $S(t)h_0$:
\begin{lemma}\label{lelinearLinfty}
	Assume $\left\{S(t)h_0\right\}(t,x,v)=h(t,x,v,I)$ be the solution for \eqref{LBEh} with the initial data $f(0,x,v,I)=f_0(x,v,I)$ satisfying \eqref{M}, \eqref{J} \eqref{E} and $(M_0,J_0,E_0)=0$, then for $\be>5$, there exists constants $\la$ and $C$ such that
	\begin{align}\label{liinfde}
	\|S(t)h_0\|_{\infty}=\|h(t)\|_{\infty}\leq Ce^{-\la t}\|h_0\|_\infty,
	\end{align}
for any $t\geq 0$.
\end{lemma}
\begin{proof}
	First we let $\la$ be defined in Lemma \ref{leL2}. Actually we can choose $\la$ so small such that $0<\la\leq \frac{\nu_0}{2}$ where $\nu_0$ is the lower bound of $\nu$. Using \eqref{mildlinear} and recalling the definition of $k_w$ \eqref{defkw}, we have
	\begin{align*}
		h(t,x,v,I)&=e^{-\nu(v,I)t}h_0(x-vt,v,I)\notag\\
		&\qquad+\int_0^t e^{-\nu(v,I)(t-s)}\int_{\R^3\times \R_+}k_w(v,v_*,I,I_*)h(s,x-v(t-s),v_*,I_*)dv_*dI_* ds\notag\\
		&=e^{-\nu(v,I)t}h_0(x-vt,v,I)\notag\\
		&\qquad+\int_0^t e^{-\nu(v,I)(t-s)}\int_{\R^3\times \R_+}k_w(v,v_*,I,I_*)e^{-\nu(v_*,I_*)s}h_0(x-v(t-s),v_*,I_*)dv_*dI_* ds\notag\\
		&\qquad+\int_0^t e^{-\nu(v,I)(t-s)}\int_{\R^3\times \R_+}k_w(v,v_*,I,I_*)\int_0^s e^{-\nu(v_*,I_*)(s-s_1)}\notag\\
		&\qquad\quad\times\int_{\R^3\times \R_+}k_w(v_*,v_{**},I_*,I_{**})h(s_1,x_1-v_*(s-s_1),v_{**},I_{**})dv_{**}dI_{**}ds_1dv_*dI_* ds,
	\end{align*}
where $x_1=x-v(t-s)$. Then it holds directly that
\begin{align*}
	|h(t,x,v,I)|\leq e^{-\nu_0t}\|h_0\|_\infty+H_1+H_2,
\end{align*}
where $\nu_0$ is defined in \eqref{estnu},
\begin{align*}
	H_1=\|h_0\|_\infty\int_0^t e^{-\nu(v,I)(t-s)}\int_{\R^3\times \R_+}k_w(v,v_*,I,I_*)e^{-\nu(v_*,I_*)s}dv_*dI_* ds,
\end{align*}
and
\begin{align*}
	H_2&=\int_0^t e^{-\nu(v,I)(t-s)}\int_{\R^3\times \R_+}k_w(v,v_*,I,I_*)\int_0^s e^{-\nu(v_*,I_*)(s-s_1)}\notag\\
	&\qquad\quad\times\int_{\R^3\times \R_+}k_w(v_*,v_{**},I_*,I_{**})|h(s_1,x_1-v_*(s-s_1),v_{**},I_{**})|dv_{**}dI_{**}ds_1dv_*dI_* ds.
\end{align*}
We apply \eqref{estnu} and \eqref{estk} to get
\begin{align}\label{H1}
	H_1&\leq\|h_0\|_\infty\int_0^t e^{-\nu_0(t-s)}\int_{\R^3\times \R_+}k_w(v,v_*,I,I_*)e^{-\nu_0s}dv_*dI_* ds\notag\\
	&\leq Cte^{-\nu_0t}\|h_0\|_\infty\int_{\R^3\times \R_+}k_w(v,v_*,I,I_*)dv_*dI_*\notag\\
	&\leq Ce^{-\la t}\|h_0\|_\infty.
\end{align}
For $H_2$, we divided it into five cases. First by Fubini's theorem, we rewrite
\begin{align*}
	H_2&=\int_0^t\int_0^s\int_{(\R^3)^2\times (\R_+)^2} e^{-\nu(v,I)(t-s)}e^{-\nu(v_*,I_*)(s-s_1)}k_w(v,v_*,I,I_*) \notag\\
	&\qquad\quad\times k_w(v_*,v_{**},I_*,I_{**})|h(s_1,x_1-v_*(s-s_1),v_{**},I_{**})|dv_{**}dv_*dI_{**}dI_* ds_1ds.
\end{align*}
Define 
\begin{align*}
	\chi_1=\chi_{\{|v|\geq N\}}&+\chi_{\{I\geq N\}},\
	  \chi_2=\chi_{\{|v|\leq N,|v_*|\geq 2N\}}+\chi_{\{|v_*|\leq 2N,|v_{**}|\geq 3N\}},\ \chi_3=\chi_{\{I_*\geq N\}}+\chi_{\{I_{**}\geq N\}},\notag\\
&\chi_4=\chi_{\{s-s_1\leq\ka\}},\quad \chi_5=\chi_{\{|v|\leq N,|v_*|\leq 2N, |v_{**}|\leq3N\}}\chi_{\{I,I_*,I_{**}\leq N\}}\chi_{\{s-s_1\geq \ka\}}.
\end{align*}
We have
\begin{align*}
		H_2&\leq\sum^5_{i=0}\int_0^t\int_0^s\int_{(\R^3)^2\times (\R_+)^2} \chi_ie^{-\nu(v,I)(t-s)}e^{-\nu(v_*,I_*)(s-s_1)}k_w(v,v_*,I,I_*) \notag\\
	&\qquad\quad\times k_w(v_*,v_{**},I_*,I_{**})|h(s_1,x_1-v_*(s-s_1),v_{**},I_{**})|dv_{**}dv_*dI_{**}dI_* ds_1ds\notag\\
	&=\sum^5_{i=0}H_{2i}.
\end{align*}
Consider $H_{21}$, which corresponds to the case that $|v|\geq N$ or $I\geq N$. Using \eqref{estnu}, one has
\begin{align}\label{1H21}
	H_{21}&\leq \int_0^t\int_0^s\int_{(\R^3)^2\times (\R_+)^2} \chi_1e^{-\nu_0(t-s)}e^{-\nu_0(s-s_1)}k_w(v,v_*,I,I_*) \notag\\
	&\qquad\quad\times k_w(v_*,v_{**},I_*,I_{**})\|h(s_1)\|_{\infty}dv_{**}dv_*dI_{**}dI_* ds_1ds\notag\\
	&\leq Ce^{-\la t}\sup_{0\leq s\leq t}\|e^{\la s}h(s)\|_{\infty}\int_0^t\int_0^s\int_{(\R^3)^2\times (\R_+)^2} \chi_1e^{-\frac{\nu_0}{2} (t-s_1)}k_w(v,v_*,I,I_*) \notag\\
	&\qquad\quad\times k_w(v_*,v_{**},I_*,I_{**})dv_{**}dv_*dI_{**}dI_* ds_1ds\notag\\
	&\leq Ce^{-\la t}\sup_{0\leq s\leq t}\|e^{\la s}h(s)\|_{\infty}\notag\\
	&\qquad\quad\times\int_{(\R^3)^2\times (\R_+)^2}\chi_1 k_w(v,v_*,I,I_*)k_w(v_*,v_{**},I_*,I_{**})dv_{**}dv_*dI_{**}dI_*.
\end{align}
It follows from \eqref{estk} that
\begin{align}\label{12kw}
	&\int_{(\R^3)^2\times (\R_+)^2}\chi_1 k_w(v,v_*,I,I_*)k_w(v_*,v_{**},I_*,I_{**})dv_{**}dv_*dI_{**}dI_*\notag\\
	&\leq C\int_{\R^3\times \R_+}\chi_1 k_w(v,v_*,I,I_*)dv_*dI_*\notag\\
	&\leq C\frac{\chi_{\{|v|\geq N\}}+\chi_{\{I\geq N\}}}{1+|v|+I^{1/8}}\notag\\
	&\leq \frac{1}{N^{1/8}}.
\end{align}
We conclude from \eqref{1H21} and \eqref{12kw} that
\begin{align}\label{H21}
	H_{21}&\leq \frac{C}{N^{1/8}}e^{-\la t}\sup_{0\leq s\leq t}\|e^{\la s}h(s)\|_{\infty}.
\end{align}
For $H_{22}$, we first notice that $|v|\leq N$, $|v_*|\geq 2N$ implies $|v-v_*|\geq N$, and $|v_*|\leq 2N$, $|v_{**}|\geq 3N$ implies $|v_*-v_{**}|\geq N$. Then similar arguments in \eqref{1H21} and \eqref{12kw} show that
\begin{align}\label{1H22}
	H_{22}&\leq Ce^{-\la t}\sup_{0\leq s\leq t}\|e^{\la s}h(s)\|_{\infty}\notag\\
	&\qquad\quad\times\int_{(\R^3)^2\times (\R_+)^2}\chi_2 k_w(v,v_*,I,I_*)k_w(v_*,v_{**},I_*,I_{**})dv_{**}dv_*dI_{**}dI_*.
\end{align}
Again use \eqref{estk} to obtain
\begin{align}\label{21kw}
	&\int_{(\R^3)^2\times (\R_+)^2}\chi_2 k_w(v,v_*,I,I_*)k_w(v_*,v_{**},I_*,I_{**})dv_{**}dv_*dI_{**}dI_*\notag\\
	&=\int_{(\R^3)^2\times (\R_+)^2}\chi_2 k_w(v,v_*,I,I_*)e^{\frac{|v-v_*|^2}{64}}e^{-\frac{|v-v_*|^2}{64}}\notag\\
	&\qquad\qquad\qquad\qquad\times k_w(v_*,v_{**},I_*,I_{**})e^{\frac{|v_*-v_{**}|^2}{64}}e^{-\frac{|v_*-v_{**}|^2}{64}}dv_{**}dv_*dI_{**}dI_*\notag\\
	&\leq\frac{C}{N^{1/8}}\int_{(\R^3)^2\times (\R_+)^2} k_w(v,v_*,I,I_*)e^{\frac{|v-v_*|^2}{64}}k_w(v_*,v_{**},I_*,I_{**})e^{\frac{|v_*-v_{**}|^2}{64}}dv_{**}dv_*dI_{**}dI_*\notag\\
	&\leq \frac{1}{N^{1/8}}.
\end{align}
Combining \eqref{1H22} and \eqref{21kw}, we get
\begin{align}\label{H22}
	H_{22}&\leq \frac{C}{N^{1/8}}e^{-\la t}\sup_{0\leq s\leq t}\|e^{\la s}h(s)\|_{\infty}.
\end{align}
Similarly as above, we have for $H_{23}$ that
\begin{align}\label{H23}
	H_{23}&\leq Ce^{-\la t}\sup_{0\leq s\leq t}\|e^{\la s}h(s)\|_{\infty}\notag\\
	&\ \times\int_{(\R^3)^2\times (\R_+)^2}\chi_3 k_w(v,v_*,I,I_*)\left(\frac{1+I_*}{1+I_*}\right)^{1/8}k_w(v_*,v_{**},I_*,I_{**})\left(\frac{1+I_{**}}{1+I_{**}}\right)^{1/8}dv_{**}dv_*dI_{**}dI_*\notag\\
	&\leq \frac{C}{N^{1/8}}e^{-\la t}\sup_{0\leq s\leq t}\|e^{\la s}h(s)\|_{\infty}.\notag\\
	&\qquad\times\int_{(\R^3)^2\times (\R_+)^2} k_w(v,v_*,I,I_*)(1+I_*)^{1/8}k_w(v_*,v_{**},I_*,I_{**})(1+I_{**})^{1/8}dv_{**}dv_*dI_{**}dI_*\notag\\
	&\leq \frac{C}{N^{1/8}}e^{-\la t}\sup_{0\leq s\leq t}\|e^{\la s}h(s)\|_{\infty}.
\end{align}
It holds for $H_{24}$ that
\begin{align}\label{H24}
	H_{24}&\leq \int_0^t\int_{s-\ka}^s\int_{(\R^3)^2\times (\R_+)^2} e^{-\nu_0(t-s)}e^{-\nu_0(s-s_1)}k_w(v,v_*,I,I_*) \notag\\
	&\qquad\quad\times k_w(v_*,v_{**},I_*,I_{**})\|h(s_1)\|_{\infty}dv_{**}dv_*dI_{**}dI_* ds_1ds\notag\\
	&\leq C\ka e^{-\la t}\sup_{0\leq s\leq t}\|e^{\la s}h(s)\|_{\infty}.
\end{align}
The last step is to estimate $H_{25}$, which is the case that $|v|\leq N,|v_*|\leq 2N, |v_{**}|\leq3N$, $I,I_*,I_{**}\leq N$ and $s-s_1\geq \ka$. Notice that in this case, all the velocity and internal energy variables are bounded by $N$. We can choose a smooth function $k_N=k_N(v,v_*,I,I_*)$ with compact support such that
\begin{align*}
	\sup_{|v|\leq 3N,I\leq 3N}\int_{|v_*|\leq 3N,I_*\leq 3N}|k_w(v,v_*,I,I_*)-k_N(v,v_*,I,I_*)|dv_*dI_* \leq \frac{C}{N}.
\end{align*}
In fact, such function can be constructed by removing the singularity using a smooth cut-off function to restrict $k_w$ in the region $\{|v|\leq 3N,I\leq 3N,|v_*|\leq 3N,I_*\leq 3N,|v-v_*|\geq C_N,I\geq C_N, I_*\geq C_N\}$, where $C_N$ is a small number which depends on $N$. A direct calculation shows that
\begin{align}\label{1H25}
	H_{25}&\leq \int_0^t\int_0^s\int_{(\R^3)^2\times (\R_+)^2} \chi_5e^{-\nu_0(t-s)}e^{-\nu_0(s-s_1)}\left|k_w(v,v_*,I,I_*)-k_N(v,v_*,I,I_*)\right| \notag\\
	&\qquad\quad\times k_w(v_*,v_{**},I_*,I_{**})\|h(s_1)\|_{\infty}dv_{**}dv_*dI_{**}dI_* ds_1ds\notag\\
	&+\int_0^t\int_0^s\int_{(\R^3)^2\times (\R_+)^2} \chi_5e^{-\nu_0(t-s)}e^{-\nu_0(s-s_1)}k_w(v,v_*,I,I_*) \notag\\
	&\qquad\quad\times \left|k_w(v_*,v_{**},I_*,I_{**})-k_N(v,v_*,I,I_*)\right|\|h(s_1)\|_{\infty}dv_{**}dv_*dI_{**}dI_* ds_1ds\notag\\
	&+\int_0^t\int_0^s\int_{(\R^3)^2\times (\R_+)^2} \chi_5e^{-\nu_0(t-s)}e^{-\nu_0(s-s_1)}k_N(v,v_*,I,I_*) \notag\\
	&\qquad\quad\times k_N(v_*,v_{**},I_*,I_{**})|h(s_1,x_1-v_*(s-s_1),v_{**},I_{**})|dv_{**}dv_*dI_{**}dI_* ds_1ds\notag\\
	&\leq\frac{C}{N^{1/8}}e^{-\la t}\sup_{0\leq s\leq t}\|e^{\la s}h(s)\|_{\infty}+H_3,
\end{align}
where
\begin{align*}
	H_3=&C_N\int_0^t\int_0^s\int_{(\R^3)^2\times (\R_+)^2} \chi_5e^{-\nu_0(t-s)}e^{-\nu_0(s-s_1)}\notag\\
	&\qquad\qquad\times|h(s_1,x_1-v_*(s-s_1),v_{**},I_{**})|dv_{**}dv_*dI_{**}dI_* ds_1ds.
\end{align*}
We use change of variables $y=x_1-v_*(s-s_1)$ to obtain that
\begin{align*}
	H_3&\leq C_{\be,N,\ka}\int_0^t\int_0^s \chi_5e^{-\nu_0(t-s)}e^{-\nu_0(s-s_1)}\left(1+(s-s_1)^3\right)\notag\\
	&\qquad\qquad\qquad\times\left(\int_{\T^3\times \R^3\times \R_+}|f(s_1,y,v_{**},I_{**})|dydv_{**}dI_{**}\right) ds_1ds\notag\\
	&\leq C_{\be,N,\ka}\int_0^t\int_0^s \chi_5e^{-\nu_0(t-s)}e^{-\nu_0(s-s_1)}\left(1+(s-s_1)^3\right)\notag\\
	&\qquad\qquad\qquad\times\left(\int_{\T^3\times \R^3\times \R_+}|f(s_1,y,v_{**},I_{**})|^2dydv_{**}dI_{**}\right)^\frac{1}{2} ds_1ds.
\end{align*}
Then we can apply Lemma \ref{leL2} and the fact that
$$
\|f_0\|^2=\int_{\T^3\times \R^3\times \R_+}|f_0(x,v,I)|^2dxdvdI\leq \|h_0\|^2_\infty \int_{\T^3\times \R^3\times \R_+}\frac{1}{w^2(v,I)}dxdvdI\leq C\|h_0\|^2_\infty
$$
to get
\begin{align}\label{H3}
	H_3&\leq C_{\be,N,\ka}e^{-\la t}\|f_0\|\int_0^t\int_0^se^{-\nu_0(t-s)}e^{-\nu_0(s-s_1)}\left(1+(s-s_1)^3\right) ds_1ds\notag\\
	&\leq C_{\be,N,\ka}e^{-\la t}\|h_0\|_\infty.
\end{align}
We have from \eqref{1H25} and \eqref{H3} that
\begin{align}\label{H25}
	H_{25}\leq\frac{C}{N^{1/8}}e^{-\la t}\sup_{0\leq s\leq t}\|e^{\la s}h(s)\|_{\infty}+C_{\be,N,\ka}e^{-\la t}\|h_0\|_\infty.
\end{align}
Collecting \eqref{H1}, \eqref{H21}, \eqref{H22}, \eqref{H23}, \eqref{H24} and \eqref{H25}, it holds that
\begin{align*}
	|h(t,x,v,I)|\leq C_{\be,N,\ka}e^{-\la t}\|h_0\|_\infty+\frac{C}{N^{1/8}}e^{-\la t}\sup_{0\leq s\leq t}\|e^{\la s}h(s)\|_{\infty}+C\ka e^{-\la t}\sup_{0\leq s\leq t}\|e^{\la s}h(s)\|_{\infty}.
\end{align*}
We can multiply $e^{\la t}$ on both sides of the above inequalty and take the $L^\infty$ norm, then choose $N$ large and $\ka$ small to get
\begin{align*}
	\sup_{0\leq s\leq t}\|e^{\la s}h(s)\|_{\infty}\leq C\|h_0\|_\infty.
\end{align*}
Hence, Lemma \ref{lelinearLinfty} is proven by the above inequality.
\end{proof}

\section{Nonlinear $L^\infty$ decay}
In this section, we obtain the nonlinear $L^\infty$ decay to complete the
\begin{proof}[Proof of Theorem \ref{global}]
	We rewrite the Boltzmann equation \eqref{BE} in terms of $h=h(t,x,v,I)=(wf)(t,x,v,I)$ as follows:
	\begin{align*}
		\pa_th+v\cdot \na_x h+\nu h=K_wh+w\Ga (\frac{h}{w},\frac{h}{w}),   \quad &\dis h(0,x,v,I)=h_0(x,v,I)=(wf_0)(x,v,I).
	\end{align*}
	Recalling the definition of $S(t)$ \eqref{mildlinear} which solves the linear problem \eqref{LBEh}, we have
	\begin{align*}
		h(t,x,v,I)&=\left\{S(t)h_0\right\}(t,x,v)+\int^t_0\left\{S(t-s)w\Ga(\frac{h}{w},\frac{h}{w})(s)\right\}(s,x-v(t-s),v,I)ds\notag\\
		&=\left\{S(t)h_0\right\}(t,x,v)+\int^t_0e^{-\nu(v,I)(t-s)}\left\{w\Ga(\frac{h}{w},\frac{h}{w})(s)\right\}(s,x-v(t-s),v,I)ds\notag\\
		&\qquad+\int^t_0\int^t_se^{-\nu(v,I)(t-s_1)}K_w\left\{S(s_1-s)w\Ga(\frac{h}{w},\frac{h}{w})(s)\right\}(s,x-v(t-s),v,I)d{s_1}ds.
	\end{align*}
	Then it holds by \eqref{liinfde} that
	\begin{align}\label{esth1}
		\|h(t)\|_\infty\leq Ce^{-\la t}\|h_0\|_\infty+J_1+J_2,
	\end{align}
	where
	\begin{align*}
		J_1=\int^t_0e^{-\nu(v,I)(t-s)}\left|\left\{w\Ga(\frac{h}{w},\frac{h}{w})(s)\right\}(s,x-v(t-s),v,I)\right| ds,
	\end{align*}
and
\begin{align*}
	J_2=&\int^t_0\int^t_se^{-\nu(v,I)(t-s_1)}\int_{\R^3\times\R_+}k_w(v,v_*,I,I_*)\notag\\
	&\qquad\times\left|\left\{S(s_1-s)w\Ga(\frac{h}{w},\frac{h}{w})(s)\right\}(s,x-v(t-s),v_*,I_*)\right|dv_*dI_*d{s_1}ds.
\end{align*}
By \eqref{estGa} in Lemma \ref{lebound}, one has
\begin{align}\label{J1}
	J_1&\leq C\int^t_0e^{-\nu(v,I)(t-s)}\nu(v,I)\|h(s)\|^2_\infty ds\notag\\
	&\leq Ce^{-\frac{\la}{4} t}\sup_{0\leq s\leq t}\|e^{\frac{\la}{4} s}h(s)\|^2_{\infty}\int^t_0e^{-\frac{\nu(v,I)}{2}(t-s)}\nu(v,I) ds\notag\\
	&\leq Ce^{-\frac{\la}{4} t}\sup_{0\leq s\leq t}\|e^{\frac{\la}{4} s}h(s)\|^2_{\infty}.
\end{align}
For $J_2$, we use \eqref{estGa} and \eqref{liinfde} to get
\begin{align}\label{J2}
	J_2&\leq \int^t_0\int^t_se^{-\nu(v,I)(t-s_1)}\int_{\R^3\times\R_+}k_w(v,v_*,I,I_*)\nu(v_*,I_*)\notag\\
	&\qquad\times\left|\left\{S(s_1-s)\frac{w}{\nu}\Ga(\frac{h}{w},\frac{h}{w})(s)\right\}(s,x-v(t-s),v_*,I_*)\right|dv_*dI_*d{s_1}ds\notag\\
	&\leq C\int^t_0\int^t_se^{-\nu(v,I)(t-s_1)}\int_{\R^3\times\R_+}e^{-\la(s_1-s)}k_w(v,v_*,I,I_*)\nu(v_*,I_*)\notag\\
	&\qquad\times\big\|\left\{\frac{w}{\nu}\Ga(\frac{h}{w},\frac{h}{w})\right\}(s)\big\|_\infty dv_*dI_*d{s_1}ds\notag\\
	&\leq Ce^{-\frac{\la}{4} t}\sup_{0\leq s\leq t}\|e^{\frac{\la}{4} s}h(s)\|^2_{\infty}\int^t_0\int^t_s\int_{\R^3\times\R_+}e^{-\frac{\nu(v,I)}{2}(t-s_1)}e^{-\frac{\la}{2}(s_1-s)}\notag\\
	&\qquad\times k_w(v,v_*,I,I_*)\nu(v_*,I_*)dv_*dI_*d{s_1}ds\notag\\
	&\leq Ce^{-\frac{\la}{4} t}\sup_{0\leq s\leq t}\|e^{\frac{\la}{4} s}h(s)\|^2_{\infty}.
\end{align}
In the last inequality above, we use the fact that
\begin{align}\label{J22}
	&\int^t_0\int^t_s\int_{\R^3\times\R_+}e^{-\frac{\nu(v,I)}{2}(t-s_1)}e^{-\frac{\la}{2}(s_1-s)}k_w(v,v_*,I,I_*)\nu(v_*,I_*)dv_*dI_*d{s_1}ds\notag\\
	&=\int^t_0\int_{\R^3\times\R_+}e^{-\frac{\nu(v,I)}{2}(t-s_1)}\left(\int^{s_1}_0e^{-\frac{\la}{2}(s_1-s)}ds\right)k_w(v,v_*,I,I_*)\nu(v_*,I_*)dv_*dI_*d{s_1}\notag\\
	&\leq C\int^t_0\int_{\R^3\times\R_+}e^{-\frac{\nu(v,I)}{2}(t-s_1)}k_w(v,v_*,I,I_*)\nu(v_*,I_*)dv_*dI_*d{s_1}\notag\\
	&\leq C\int_{\R^3\times\R_+}k_w(v,v_*,I,I_*)\frac{\nu(v_*,I_*)}{\nu(v,I)}dv_*dI_*\notag\\
	&\leq C.
\end{align}
The last inequality above holds by Lemma \ref{estnu} and Lemma \ref{lemmaK}.
Combining \eqref{esth1}, \eqref{J1} and \eqref{J2}, it holds that
	\begin{align*}
	\|h(t)\|_\infty\leq Ce^{-\la t}\|h_0\|_\infty+Ce^{-\frac{\la}{4} t}\sup_{0\leq s\leq t}\|e^{\frac{\la}{4} s}h(s)\|^2_{\infty}.
\end{align*}
Multiplying $e^{\frac{\la}{4} t}$ on both sides of the above inequality, we obtain
\begin{align*}
	\sup_{0\leq s\leq t}\|e^{\frac{\la}{4} s}h(s)\|_{\infty}\leq C\|h_0\|_\infty+C\sup_{0\leq s\leq t}\|e^{\frac{\la}{4} s}h(s)\|^2_{\infty}.
\end{align*}
Therefore, we obtain \eqref{GE} and Theorem \ref{global} is proven.
\end{proof}

\medskip
\noindent {\bf Acknowledgments:}\,
The research of RJD is partially supported by the General Research Fund (Project No.~14301720) from RGC of Hong Kong and a Direct Grant from CUHK. The research of ZGL is supported by the Hong Kong PhD Fellowship Scheme.

\end{document}